\documentclass[12pt, leqno]{amsart}

\usepackage{float}
\usepackage{graphics}
\usepackage{amssymb,amsmath,amsthm,amscd}
\usepackage{mathrsfs}
\usepackage{enumerate}
\usepackage{color}
\usepackage[vcentermath,enableskew,noautoscale]{youngtab}
\usepackage[all]{xy}
\usepackage{verbatim}

\numberwithin{equation}{section}

\newcommand{\C}{{\mathbf C}}

\newcommand{\Z}{{\mathbf Z}}
\newcommand{\B}{{\mathbf{B}}}
\newcommand{\A}{{\mathbf A}}
\newcommand{\F}{{\mathbf F}}

\newcommand{\V}{{\mathbf V}}

\newcommand{\gl}{{\mathfrak{gl}}}
\newcommand{\q}{{\mathfrak{q}}}

\newcommand{\seteq}{\mathbin{:=}}
\newcommand{\uqqn}{U_q(\mathfrak{q}(n))}

\newcommand{\readw}{\operatorname{read}}
\newcommand{\sh}{\operatorname{sh}}

\theoremstyle{plain}
\newtheorem{lemma}{Lemma}[section]
\newtheorem{prop}[lemma]{Proposition}
\newtheorem{theorem}[lemma]{Theorem}
\newcommand{\Prop}{\begin{prop}}
\newcommand{\enprop}{\end{prop}}
\newcommand{\Lemma}{\begin{lemma}}
\newcommand{\enlemma}{\end{lemma}}
\newcommand{\Th}{\begin{theorem}}
\newcommand{\enth}{\end{theorem}}
\newtheorem{corollary}[lemma]{Corollary}
\newcommand{\Cor}{\begin{corollary}}
\newcommand{\encor}{\end{corollary}}
\newtheorem{definition}[lemma]{Definition}

\newcommand{\Def}{\begin{definition}}
\newcommand{\edf}{\end{definition}}
\newcommand{\ten}{10}
\newcommand{\eleven}{11}
\newcommand{\twelve}{12}
\newcommand{\thirteen}{13}

\theoremstyle{definition}
\newtheorem{remark}[lemma]{Remark}
\newtheorem{example}[lemma]{Example}

\newcommand{\g}{{\mathfrak{g}}}
\newcommand{\Uq}{{U_q(\mathfrak{q}(n))}}
\newcommand{\qn}{{\mathfrak{q}(n)}}

\newcommand{\isoto}[1][]{\mathop{\xrightarrow[#1]%
{{\raisebox{-.6ex}[0ex][-.6ex]{$\mspace{2mu}\sim\mspace{2mu}$}}}}}
\newcommand{\tensor}{\otimes}

\newcommand{\eq}{\begin{eqnarray}}
\newcommand{\eneq}{\end{eqnarray}}

\newcommand{\eqn}{\begin{eqnarray*}}
\newcommand{\eneqn}{\end{eqnarray*}}

\newcommand{\on}{\operatorname}

\newcommand{\bni}{\be[{\rm(i)}]}
\newcommand{\bna}{\be[{\rm(a)}]}

\newcommand{\QED}{\end{proof}}
\newcommand{\Proof}{\begin{proof}}

\newcommand{\soplus}{\mathop{\mbox{\normalsize$\bigoplus$}}\limits}

\newcommand{\cl}{\colon}

\newcommand{\ba}{\begin{array}}
\newcommand{\ea}{\end{array}}

\newcommand{\bi}{\begin{enumerate}[{\rm(i)}]}

\newcommand{\set}[2]{\left\{#1 \mathbin{;} #2 \right\}}

\newcommand{\hs}{\hspace*}

\newcommand{\eqsub}{\begin{subequations}\begin{eqnarray}}
\newcommand{\eneqsub}{\end{eqnarray}\end{subequations}}

\newcommand{\ol}{\overline}

\newcommand{\nc}{\newcommand}
\nc{\la}{\lambda}
\nc{\lam}{\lambda}
\nc{\U}[1][\g]{U_q(#1)}
\nc{\te}{\tilde{e}}
\nc{\tei}{\tilde{e}_i}
\nc{\tf}{\tilde{f}}
\nc{\tfi}{\tilde{f}_i}
\nc{\tU}{\widetilde U_q(\g)}
\nc{\tE}{\tilde{E}}
\nc{\tF}{\tilde{F}}

\nc{\tk}{\tilde{k}}
\nc{\tkone}{\tk_{\ol{1}}}
\nc{\teone}{\tilde{e}_{\ol{1}}}
\nc{\tfone}{\tilde{f}_{\ol{1}}}

\nc{\teibar}{\tilde{e}_{\ol{i}}} \nc{\tfibar}{\tilde{f}_{\ol{i}}}
\nc{\tki}{{\tk}_{\ol {i}}}

\nc{\BZ}{{\mathbb{Z}}}
\nc{\al}{\alpha}
\nc{\qs}{{q}}
\nc{\lan}{\langle}
\nc{\ran}{\rangle}
\nc{\re}{{\mathrm{re}}}
\nc{\wt}{\operatorname{wt}}
\nc{\ch}{\operatorname{ch}}
\nc{\Uf}[1][\g]{U^-_q(#1)}
\nc{\Ue}{U^+_q(\g)}
\nc{\eps}{\varepsilon}
\nc{\vphi}{\varphi}
\nc{\sphi}{\varphi^*}
\nc{\seps}{\varepsilon^*}

\nc{\nn}{\nonumber}
\def\max{{\mathop{\mathrm{max}}}}
\nc{\vp}{\varpi}
\nc{\cls}{{\operatorname{cl}}}
\nc{\Wt}{{\operatorname{Wt}}}
\nc{\Us}{U'_q(\g)}
\nc{\La}{\Lambda}
\nc{\ro}{{\rm(}}
\nc{\rf}{{\rm)}}
\nc{\norm}{{\mathrm{norm}}}
\nc{\qbox}{\quad\mbox}
\nc{\braid}{{\mathfrak{B}}}
\nc{\Ad}{\operatorname{Ad}}
\nc{\Aut}{\operatorname{Aut}}
\nc{\dt}[1]{\tilde{\tilde #1}}
\nc{\Sn}{S^{{\mathrm{norm}}}}
\nc{\aff}{{\mathrm{aff}}}
\nc{\rk}{{\mathrm{rk}}}
\nc{\tQ}{\widetilde{Q}}
\nc{\tP}{\widetilde{P}}
\nc{\tW}{\widetilde{W}}
\nc{\Dyn}{\mathrm{Dyn}}
\nc{\tD}{\widetilde{\Delta}}
\nc{\height}{{\operatorname{ht}}}
\nc{\bl}{\bigl}
\nc{\br}{\bigr}
\nc{\Hecke}{\mathrm{H}}
\nc{\HA}{\Hecke^{\mathrm{A}}}
\nc{\HB}{\Hecke^{\mathrm{B}}}
\nc{\K}{\mathrm{K}}
\newcommand{\scbul}{{\,\raise1pt\hbox{$\scriptscriptstyle\bullet$}\,}}
\nc{\vac}{{\phi}}
\nc{\Bt}{\B_\theta(\g)}
\nc{\be}{\begin{enumerate}}
\nc{\ee}{\end{enumerate}}
\nc{\low}{{\mathrm{low}}}
\nc{\upper}{{\mathrm{up}}}
\nc{\Zodd}{\Z_{\mathrm{odd}}}
\nc{\Ft}[1][n]{\mathbb{P}\mathrm{ol}_{#1}}
\nc{\Ftf}[1][n]{\widetilde{\mathbb{P}\mathrm{ol}}_{#1}}
\nc{\KA}{\on{K}^{\mathrm{A}}}
\nc{\KB}{\on{K}^{\mathrm{B}}}
\nc{\Res}{\on{Res}}
\nc{\Fc}[1][{n,m}]{\mathbf{F}_{#1}}
\nc{\tphi}{\tilde{\varphi}}
\nc{\CO}{\mathscr{O}}
\nc{\inte}{\mathrm{int}}
\nc{\Oint}{\mathcal{O}^{\ge0}_{\inte}}
\nc{\vs}{\vspace}
\nc{\tL}{\widetilde{L}}
\nc{\noi}{\noindent}
\nc{\heigh}{\mathfrak{t}}
\nc{\lowest}{\mathfrak{l}}

\newenvironment{rouge}
{\color{red}}
{}

\nc{\bred}{\begin{rouge}}
\nc{\ered}{\end{rouge}}

\newlength{\mylength}
\title[Queer crystals and semistandard decomposition tableaux]{Crystal bases for the quantum queer superalgebra \\ and
semistandard decomposition tableaux}

\author[D. Grantcharov, J. H. Jung, S.-J. Kang, M. Kashiwara, M. Kim]{Dimitar Grantcharov$^{1}$, Ji Hye Jung$^{2}$, Seok-Jin
Kang$^{3}$, \\ Masaki Kashiwara$^{4}$, Myungho Kim$^{5}$}

\address{Department of Mathematics \\
         University of Texas at Arlington \\ Arlington, TX 76021, USA}

         \email{grandim@uta.edu}

\address{Department of Mathematical Sciences \\
         Seoul National University \\ Seoul 151-747, Korea}

         \email{jhjung@math.snu.ac.kr}

\address{Department of Mathematical Sciences
         and
         Research Institute of Mathematics \\
         Seoul National University \\ Seoul 151-747, Korea}

         \email{sjkang@math.snu.ac.kr}

\address{Research Institute for Mathematical Sciences \\
          Kyoto University \\ Kyoto 606-8502, Japan \\
          \& Department of Mathematical Sciences \\
         Seoul National University \\ Seoul 151-747, Korea}

         \email{masaki@kurims.kyoto-u.ac.jp}

\address{Department of Mathematical Sciences \\
         Seoul National University \\ Seoul 151-747, Korea}
         \email{mkim@math.snu.ac.kr}

\thanks{$^{1}$This work was partially supported by NSA grant H98230-10-1-0207 and by Research Institute for Mathematical Sciences, Kyoto University.}
\thanks{$^{2}$This work was partially supported by BK21 Mathematical Sciences Division and by NRF Grant \# 2010-0010753.}
\thanks{$^{3}$This work was partially supported by KRF Grant \# 2007-341-C00001, by  NRF Grant \# 2010-0010753 and by NRF Grant \# 2010-0019516.} 
\thanks{$^{4}$This work was partially supported by Grant-in-Aid for Scientific Research (B) 23340005,
Japan Society for the Promotion of Science.}
\thanks{$^{5}$This work was partially supported by KRF Grant \# 2007-341-C00001 and
by NRF Grant \# 2010-0019516.}


\begin{document}

\maketitle

\begin{abstract}
In this paper, we give an explicit combinatorial realization of the
crystal $B(\la)$ for an irreducible highest weight $U_q(\qn)$-module
$V(\la)$ in terms of semistandard decomposition tableaux. We present
an insertion scheme for semistandard decomposition tableaux and give
algorithms of decomposing the tensor product of $\q(n)$-crystals.
Consequently, we obtain explicit combinatorial descriptions of the
shifted Littlewood-Richardson coefficients.
\end{abstract}

\section*{Introduction}

The queer Lie superalgebra ${\mathfrak q} (n)$ is the second super
analogue of the general linear Lie algebra $\mathfrak{gl} (n)$ and
is one of the most interesting algebraic structures studied both by
mathematicians and physicists. It has been known since its inception
that the representation theory of ${\mathfrak q} (n)$ is  rather
complicated. This is partly due to the fact that every Cartan
subalgebra  of ${\mathfrak q} (n)$ is noncommutative, and, as a
result, the highest weight space of a highest weight ${\mathfrak q}
(n)$-module has a Clifford module structure. The combinatorial
structure of the finite-dimensional $\q(n)$-modules is especially
interesting. The {\it tensor representations} of ${\mathfrak q}(n)$;
i.e., those that appear as  subrepresentations of tensor powers of
the vector representation, are involved in the queer analogue of the
celebrated {\it Schur-Weyl duality}, often referred to as the {\it
Schur-Weyl-Sergeev duality}. This duality  was obtained in
\cite{Ser} for $U (\qn)$ and in \cite{Ol} for $U_q (\qn)$.
The isomorphism 
classes of tensor representations of
${\mathfrak q}(n)$ are parametrized by the set $\Lambda^+$
of strict partitions of $n$. Another
important property of the tensor representations is that their
characters are multiples of the so-called {\it Schur's
$Q$-functions} \cite{Ser}.

The combinatorial description of the tensor modules over $U_q (\qn)$
can best be understood in the language of {\it crystal bases}.
Originally introduced in \cite{Kas90, Kas91} for integrable modules
over quantum groups associated with symmetrizable Kac-Moody
algebras, the crystal basis theory is considered nowadays as one of
the most prominent discoveries in the combinatorial representation
theory. Some of the significant features of the crystal bases
include: an extremely nice behavior with respect to tensor products
and important connections with combinatorics of Young tableaux and
Young walls (see, for instance, \cite{BKK, Ka2003, KK08, KN, MM,
N93}).

The numerous combinatorial applications of the queer Lie
superalgebra together with the fact that the category  of tensor
representations is semisimple raise a natural question - is there a
crystal basis theory for this category? The answer to this question
is affirmative and the solution has been recently obtained in two
steps. First, the highest weight module theory for
$U_q(\mathfrak{q}(n))$ was developed in \cite{GJKK}. Second, the
crystal basis theory for the category $\Oint$ of tensor
representations over $U_q(\mathfrak{q}(n))$ was established in
\cite{GJKKK}. As one can expect, due to the queer nature of
$\mathfrak{q} (n)$, the definition of a crystal basis for modules in
$\Oint$ is different from the one used for all other categories of
representations studied so far. A {\it crystal basis} for a
$\Uq$-module $M$ in the category $\Oint$ is defined to be a triple
$(L,B, (l_b)_{b \in B})$, where the crystal lattice $L$ is a free
$\C[[q]]$-submodule of $M$, $B$ is a finite $\gl(n)$-crystal,
$(l_b)_{b \in B}$ is a family of nonzero vector spaces such that
$L/qL = \soplus_{b \in B} l_b$, with a set of compatibility
conditions for the action of the Kashiwara operators. The main
result in \cite{GJKKK} is the existence and uniqueness theorem for a
crystal basis of any module in the category $\Oint$. The crystal
corresponding to an irreducible highest weight module $V(\lambda)$
is denoted by $B(\lambda)$.

Once a crystal basis theory for the category $\Oint$ is established,
the next task is to look at the following two problems:
\vskip 2mm
\bna
\item to find an explicit combinatorial realization of the crystal
$B(\lambda)$,

\item to establish a Littlewood-Richardson rule for the tensor
product of crystals $B(\lambda) \otimes B(\mu)$.
\ee

\vskip 2mm \noindent The purpose of this paper is to solve these
problems. A class of combinatorial objects that describe the tensor
representations of ${\mathfrak q}(n)$ has been known for more than
thirty years - the {\it shifted semistandard Young tableaux}. These
objects have been extensively studied  by Sagan, Stembridge, Worley,
and others, leading to important and deep results (in particular,
the shifted Littlewood-Richardson rule) \cite{Sag, St, Wor}.
Unfortunately, the set of  shifted semistandard Young tableaux of
fixed shape  does not have a natural crystal structure, and for this
reason we use another setting.

Our approach is based on the notion of semistandard decomposition
tableaux, which was first introduced by Serrano \cite{Serra}. For
our purpose, we slightly change the definition used in \cite{Serra}.
In our setting, a {\it hook word} is a word $u = u_1 \cdots  u_N$
for which $u_1 \ge u_2 \ge \cdots \ge u_k <u_{k+1} <\cdots < u_N$
for some $k$. Then, a \emph{semistandard decomposition tableau} is a
filling $T$ of a shifted shape $\lambda = (\lambda_1, \ldots,
\lambda_n)$ with elements of $\{1,2,\ldots,n\}$ such that

\ \ (i) the word $v_i$ formed by reading the $i$-th row from left to
right is a hook word of length $\lambda_i$,

\ \ (ii) $v_i$ is a hook subword of maximal length in $v_{i+1} v_i$
for $ 1 \le i \le r-1$, where $r$ is the number of nonzero
$\la_i$'s.

Our first main result is that the set $\B (\lambda)$ of all
semistandard decomposition tableaux of shifted shape $\lambda$ has a
crystal structure and is isomorphic to $B(\lambda)$. This
combinatorial realization of crystals and the properties of lowest
weight vectors enable us to decompose the tensor product $\B
(\lambda) \otimes \B (\mu)$ into a disjoint union of connected
components as follows:
\begin{equation} \label{eq_decomposition}
\B(\la) \otimes \B(\mu) \simeq \soplus_{\nu \in \Lambda^+} \B(\nu)^{\oplus f^{\nu}_{\la, \mu}},
\end{equation}
where $f^{\nu}_{\la, \mu} = |LR_{\la,\mu}^{\nu}|$ and
\begin{equation*}
\ba{l}
LR_{\la,\mu}^{\nu} = \{u=u_1\cdots u_N \in \B(\la) \ ;  {\rm (a)} \ \wt(u)= w_0(\nu -\mu) \ \text{and} \\
\hs{15ex}  {\rm (b)} \ \mu + \epsilon_{n-u_N+1} +  \cdots +
\epsilon_{n-u_{k}+1} \in  \Lambda^+ \
  \text{for all} \ 1 \leq k \leq N\}.
\end{array}
\end{equation*}
We call $f^{\nu}_{\la, \mu}$ the \emph{shifted Littlewood-Richardson coefficient}.

As seen in \eqref{eq_decomposition}, the connected component
containing $T \otimes T' \in \B(\la) \otimes \B(\mu)$ is isomorphic
to $\B(\nu)$ for some $\nu$. In order to find $\nu$ and the element
$S$ of $\B(\nu)$ 
corresponding to $T \otimes T'$ explicitly, we consider the {\it
insertion scheme} for semistandard decomposition tableaux. Namely,
for a semistandard decomposition tableau $T$ and a letter $x$, we
define the filling $T \leftarrow x$ and prove that $T \leftarrow x$
is a semistandard decomposition tableau 
and that $S=T\leftarrow x$ and $\nu=\sh(T \leftarrow x)$
the shape of $T \leftarrow x$. From here one easily
defines $T \leftarrow T'$ for semistandard decomposition tableaux
$T$ and $T'$. Our insertion scheme is analogous to the one
introduced in \cite{Serra} and can be considered as a variation of
those used for shifted tableaux by Fomin, Haiman, Sagan,  and
Worley, \cite{Fom, Hai, Sag, Wor}. It turns out that there exists a
crystal isomorphism between the connected component containing 
$T \otimes T'$ and the crystal $\B(\sh(T \leftarrow T'))$ sending $T \otimes
T'$ to $T \leftarrow T'$. A crucial part of the proof is a queer version of
the
\emph{Knuth relation}, which is a crystal isomorphism between
certain sets of four-letter words. Using this insertion scheme and
the properties of lowest weight vectors, we obtain the following
decomposition:
$$\B (\lambda) \otimes \B (\mu) \simeq \soplus_{\stackrel{T \in \B{(\la)} ;}
{T \leftarrow L^{\mu} = L^{\nu} \text{for some} \ \nu}} \B (\sh(T \leftarrow L^{\mu})).$$
Here, $L^{\nu}$ denotes a unique lowest weight vector in $\B(\nu)$.

Finally, we introduce the notion of \emph{recording tableaux} of the
insertion scheme. The recording tableaux characterize the connected
components in $\B^{\otimes N}$ 
or $\B(\la) \otimes \B(\mu)$. That
is, any two elements are in the same connected component if and only
if they have the same recording tableau. For the insertion
$P(u)=(\cdots ((u_1 \leftarrow u_2) \leftarrow u_3) \cdots )
\leftarrow u_N$ of the word $u=u_1 u_2 \cdots u_N \in \B^{\otimes
N}$, the recording tableau $Q(u)$ is defined to be the {\it standard
shifted tableau} of the same shape as $P(u)$ which records the newly
added boxes with $1, 2, \ldots, N$. One can show that to each
standard shifted tableau there is a unique lowest weight vector in
$\B^{\otimes N}$. Hence the multiplicity of $\B (\lambda)$ in
$\B^{\otimes N}$ is equal to the number of standard shifted tableaux
of shape $\lambda$, which is denoted by $f^{\la}$.

On the other hand, when $u_1u_2 \cdots u_N$ is the reading word of
$T$, we define the insertion \mbox{$T \rightarrow T'$} by \mbox{$u_1
\leftarrow (u_2 \leftarrow \cdots \leftarrow (u_{N-1}\leftarrow (u_N
\leftarrow T')) \cdots )$}. Then the corresponding recording tableau
$Q=Q(T \rightarrow T')$ is defined by the following conditions: \bna
\item $Q$ is a standard shifted tableau of shape $\nu / \mu$,
where $\nu =\sh(T \rightarrow T')$ and $\mu = \sh(T')$,
\item $(n-r_{|\la|}+1) \otimes (n-r_{|\la|-1}+1) \otimes \cdots \otimes (n-r_1 +1)$ is a
semistandard decomposition tableau of shifted shape $\la$, where
$\la=\sh(T)$ and $r_k$ denotes the row of the entry $k$ in $Q$. \ee
Such a tableau $Q$ is called the \emph{shifted Littlewood-Richardson
tableau of shape $\nu / \mu$ and type $\la$} and we show that there
is a 1-1 correspondence between the set of shifted
Littlewood-Richardson tableaux of shape $\nu / \mu$ and type $\la$
and the set of lowest weight vectors of weight $w_{0}(\nu)$ in
$\B(\la) \otimes \B(\mu)$. It follows that the number of the shifted
Littlewood-Richardson tableaux of shape $\nu / \mu$ and type $\la$
 gives the shifted Littlewood-Richardson coefficient $f^{\nu}_{\la, \mu}$.

This paper is organized as follows. In Section 1, we collect some
important definitions and facts on the crystal basis theory for the
$U_q({\mathfrak q}(n))$-modules in the category $\Oint$. In Section
2, we introduce the notion of semistandard decomposition tableaux
and prove that $\B (\lambda) \simeq B(\lambda)$. Furthermore, we
give the shifted Littlewood-Richardson rule for $\B(\la) \otimes
\B(\mu)$. Section 3 is devoted to the insertion scheme and its
properties, while in the last section we present the notion of the
recording tableaux and give another description of the shifted
Littlewood-Richardson coefficients.

\medskip \noindent
{\it Acknowledgements.} We would like to thank Professor Soojin
Cho of Ajou University for many valuable and inspiring discussions
on the combinatorics of semistandard decomposition tableaux.
\vskip 3mm

\section{Crystal bases for the quantum queer superalgebra}

\subsection{The quantum queer superalgebra}\label{sec:qn}

Let $\F=\C((q))$ be the field of formal Laurent series in an
indeterminate $q$ and let $\A=\C[[q]]$ be the subring of $\F$
consisting of formal power series in $q$. For $k \in \Z_{\ge 0}$, we
define
$$[k]= \frac{q^k - q^{-k}}{q - q^{-1}}, \quad [0]!=1, \quad [k]! =
[k] [k-1] \cdots [2][1].$$

 For an integer $n \geq 2$,
let $P^{\vee} = \Z k_1 \oplus \cdots \oplus \Z k_n$ be a free
abelian group of rank $n$ and let ${\mathfrak h_{\ol 0}} = \C \otimes_{\Z}
P^{\vee}$ be the {\em even part of the Cartan subalgebra}. Define the linear functionals
$\epsilon_i \in \mathfrak{h}_{\ol 0}^*$ by $\epsilon_i(k_j) = \delta_{ij}$
$(i,j=1, \ldots, n)$ and set $P= \Z \epsilon_1 \oplus \cdots \oplus
\Z \epsilon_n$. We denote by $\alpha_i = \epsilon_i -
\epsilon_{i+1}$ the {\em simple roots} and by $h_i=k_i-k_{i+1}$
the {\em simple coroots}.

\Def The {\em quantum queer superalgebra $U_q(\mathfrak{q}(n))$} is
the superalgebra over $\F$ with $1$ generated by the symbols $e_i$,
$f_i$, $e_{\ol i}$, $f_{\ol i}$ $(i=1, \ldots, n-1)$, $q^{h}$ $(h\in
P^\vee)$, $k_{\ol j}$ $(j=1, \ldots, n)$ with the following defining
relations.
\begin{align}
\allowdisplaybreaks
\nonumber & q^{0}=1, \ \ q^{h_1} q^{h_2} = q^{h_1 + h_2} \ \ (h_1,
h_2 \in P^{\vee}), \\
\nonumber & q^h e_i q^{-h} = q^{\alpha_i(h)} e_i \ \ (h\in P^{\vee}), \displaybreak[1]\\
\nonumber & q^h f_i q^{-h} = q^{-\alpha_i(h)} f_i \ \ (h\in P^{\vee}), \displaybreak[1]\\
\nonumber & q^h k_{\ol j} = k_{\ol j} q^h, \displaybreak[1]\\
\nonumber & e_i f_j - f_j e_i = \delta_{ij} \dfrac{q^{k_i - k_{i+1}} - q^{-k_i
+ k_{i+1}}}{q-q^{-1}}, \displaybreak[1]\\
\nonumber & e_i e_j - e_j e_i = f_i f_j - f_j f_i = 0 \quad \text{if} \ |i-j|>1, \displaybreak[1]\\
\nonumber & e_i^2 e_j -(q+q^{-1}) e_i e_j e_i  + e_j e_i^2= 0  \quad \text{if} \ |i-j|=1,\displaybreak[1]\\
\nonumber & f_i^2 f_j - (q+q^{-1}) f_i f_j f_i + f_j f_i^2 = 0  \quad \text{if} \ |i-j|=1,\displaybreak[1]\\
\nonumber & k_{\ol i}^2 = \dfrac{q^{2k_i} - q^{-2k_i}}{q^2 - q^{-2}},  \displaybreak[1]\\
          & k_{\ol i} k_{\ol j} + k_{\ol j} k_{\ol i} =0 \ \ \quad \text{if} \  i \neq j, \displaybreak[1]\\
\nonumber & k_{\ol i} e_i - q e_i k_{\ol i} = e_{\ol i} q^{-k_i}, \ q k_{\ol i}e_{i-1}- e_{i-1}k_{\ol i}=-q^{-k_i} e_{\ol{i-1}}, \displaybreak[1]\\
\nonumber & k_{\ol i}e_j-e_jk_{\ol i}=0 \quad \text{if} \ j \neq i,i-1, \displaybreak[1]\\
\nonumber & k_{\ol i} f_i - q f_i k_{\ol i} = -f_{\ol i} q^{k_i}, \ q k_{\ol i} f_{i-1}-f_{i-1}k_{\ol i}=q^{k_i}f_{\ol{i-1}}, \displaybreak[1]\\
\nonumber & k_{\ol i}f_j-f_jk_{\ol i}=0 \quad \text{if} \ j \neq i, i-1, \displaybreak[1]\\
\nonumber & e_i f_{\ol j} - f_{\ol j} e_i = \delta_{ij} (k_{\ol i}
q^{-k_{i+1}} - k_{\ol{i+1}} q^{-k_i}),  \displaybreak[1]\\
\nonumber & e_{\ol i} f_j - f_j e_{\ol i} = \delta_{ij} (k_{\ol i}
q^{k_{i+1}} - k_{\ol{i+1}} q^{k_i}),  \displaybreak[1]\\
\nonumber &e_i e_{\ol i} - e_{\ol i} e_i = f_i f_{\ol i} - f_{\ol i} f_i = 0,  \displaybreak[1]\\
\nonumber &e_i e_{i+1} - q e_{i+1}e_i =
e_{\overline{i}}e_{\overline{i+1}}+ q
e_{\overline{i+1}}e_{\overline{i}},  \displaybreak[1]\\
\nonumber&q f_{i+1}f_i - f_i f_{i+1} =
f_{\overline{i}}f_{\overline{i+1}}+ q f_{\overline{i+1}}f_{\overline{i}},  \displaybreak[1]\\
\nonumber & e_i^2 e_{\overline{j}} - (q+q^{-1})e_i e_{\overline{j}}
e_i + e_{\overline{j}} e_i^2= 0 \quad \text{if} \ |i-j|=1,  \displaybreak[1]\\
\nonumber & f_i^2 f_{\overline{j}} - (q+q^{-1})f_i f_{\overline{j}}
f_i + f_{\overline{j}} f_i^2=0 \quad \text{if} \ |i-j|=1.
\end{align}
\edf
The generators $e_i$, $f_i$ $(i=1, \ldots, n-1)$, $q^{h}$ ($h\in
P^\vee$) are regarded as {\em even} and $e_{\ol i}$, $f_{\ol i}$
$(i=1, \ldots, n-1)$, $k_{\ol j}$ $(j=1, \ldots, n)$ are {\em odd}.
From the defining relations, it is easy to see that the even
generators together with $k_{\ol 1}$ generate the whole algebra
$\Uq$.

The superalgebra $U_q(\mathfrak{q}(n))$ is a bialgebra with the
comultiplication $\Delta\cl U_q(\mathfrak{q}(n)) \to
U_q(\mathfrak{q}(n)) \otimes U_q(\mathfrak{q}(n))$ defined by
\begin{equation}
\begin{aligned}
& \Delta(q^{h})  = q^{h} \otimes q^{h}\quad\text{for $h\in P^\vee$,} \\
& \Delta(e_i)  = e_i \otimes q^{-k_i + k_{i+1}} + 1 \otimes e_i \quad\text{for $i =1,\ldots n-1$,} \\
& \Delta(f_i)  = f_i \otimes 1 + q^{k_i - k_{i+1}} \otimes f_i \quad\text{for $i =1,\ldots n-1$,}\\
& \Delta(k_{\ol 1}) =k_{\ol 1}\otimes q^{k_1}+ q^{-k_1} \otimes k_{\ol 1}.
\end{aligned}
\end{equation}

Let $U^{+}$ (respectively,\ $U^{-}$) be the subalgebra of
$U_q(\mathfrak{q}(n))$ generated by $e_i$, $e_{\ol i}$
$(i=1,\ldots, n-1)$ (respectively,\ $f_i$, $f_{\ol i}$ ($i=1, \ldots, n-1$)), and
let $U^{0}$ be the subalgebra
generated by $q^{h}$ ($h\in P^\vee$) and $k_{\ol j}$ $(j=1, \ldots, n)$.
In \cite{GJKK}, it was
shown that the algebra $U_q(\mathfrak{q}(n))$ has the {\em
triangular decomposition}:
\begin{equation}
U^{-} \otimes U^{0} \otimes U^{+}\isoto
U_q(\mathfrak{q}(n)).
\end{equation}

\subsection{The category $\Oint$.}
 Hereafter, a $U_q(\mathfrak{q}(n))$-module is understood as a
$U_q(\mathfrak{q}(n))$-supermodule.
A $U_q(\mathfrak{q}(n))$-module $M$ is called a {\em weight module}
if $M$ has a weight space decomposition $M=\soplus_{\mu \in P}
M_{\mu}$, where
$$M_{\mu}\seteq  \set{ m \in M}{q^h m = q^{\mu(h)} m \ \ \text{for all} \ h
\in P^{\vee} }.$$ The {\it set of weights of} $M$ is defined to be
$$\wt(M) = \set{\mu \in P}{M_{\mu} \neq 0 }.$$

\Def A weight module $V$ is called a {\em highest weight module
with highest weight $\la \in P$} if $V_{\la}$ is finite-dimensional and
satisfies the following conditions:
\bna
\item $V$ is generated by $V_{\la}$,
\item $e_i v = e_{\ol i} v =0$ for all $v \in V_{\la}$,
$i=1, \ldots, n-1$,
\ee  \edf

As seen in \cite{GJKK}, there exists a unique irreducible highest weight module with highest
weight $\la \in P$ up to parity change, which will be denoted by
$V(\la)$.

Set
\begin{equation*}
\begin{aligned}
P^{\ge 0} = & \{ \la = \la_1 \epsilon_1 + \cdots + \la_n \epsilon_n
\in P\, ; \, \la_j \in \Z_{\ge 0} \ \ \text{for all} \ j=1, \ldots, n \}, \\
\La^{+} = & \{\la = \la_1 \epsilon_1 + \cdots + \la_n \epsilon_n \in
P^{\ge 0}\, ; \, \text{$\la_{i} \ge \la_{i+1}$ and $\la_{i}=\la_{i+1}$ implies}\\
 &  \hs{31ex}\text{$\la_{i} = \la_{i+1} = 0$ for all $i=1,
 \ldots,n-1$}\}.
\end{aligned}
\end{equation*}
Note that each element $\la \in \La^{+}$ corresponds to a {\em
strict partition} $\la = (\la_1 > \la_2 > \cdots > \la_r >0)$.
Thus we will often call $\la \in \La^{+}$ a strict partition.
For the same reason, we call $\lambda=(\la_1,\la_2,\ldots,\la_n) \in
P^{\geq 0}$ a {\it partition} if $\la_1 \geq \la_2 \geq \cdots
\geq \la_r > \la_{r+1}=\cdots =\la_n=0$. We denote $r$  by
$\ell(\la)$.

\begin{example}
Let $$\V = \soplus_{j=1}^n \F v_{j} \oplus \soplus_{j=1}^n \F
v_{\ol j}$$ be the vector representation of $U_q(\mathfrak{q}(n))$.
The action of $\Uq$ on $\V$ is given as follows:
\begin{equation}
\ba{llll}
e_iv_j=\delta_{j,i+1}v_i, &e_iv_{\ol j}=\delta_{j,i+1}v_{\ol i},
&f_iv_j=\delta_{j,i}v_{i+1},&f_iv_{\ol j}=\delta_{j,i}v_{\ol{i+1}}, \\[1ex]
 e_{\ol i}v_j=\delta_{j,i+1}v_{\ol{i}},&e_{\ol i}v_{\ol j}=\delta_{j,i+1}v_{i},&
f_{\ol i}v_j=\delta_{j,i}v_{\ol{i+1}},&f_{\ol i}v_{\ol j}=\delta_{j,i}v_{{i+1}}, \\[1ex]
q^h v_j=q^{\epsilon_j(h)} v_j, &q^h v_{\ol j}=q^{\epsilon_j(h)} v_{\ol j},
&k_{\ol i}v_j=\delta_{j,i}v_{\ol j},&k_{\ol i}v_{\ol j}=\delta_{j,i}v_{j}.
\ea
\end{equation}
Note that $\V$ is an irreducible highest weight module with highest weight $\epsilon_1$ and $\wt (\V) = \{ \epsilon_1,...,\epsilon_n\}$.
 \end{example}

\Def We define the category $\Oint$ to be the category of
finite-dimensional weight modules $M$ satisfying the following conditions:
\bna
\item  $\wt(M) \subset P^{\ge 0}$,
\item for any $\mu\in P^{\ge0}$ and $i \in \{1,\ldots, n\}$
 such that $\lan k_i,\mu\ran=0$,
we have $k_{\ol i}\vert_{M_\mu}=0$.
\ee
\edf

\begin{prop}[{\cite[Corollary 1.12]{GJKKK2}}] \label{cor:Vtens} \hfill 
\bna
\item The abelian category $\Oint$ is semisimple.
\item
Any irreducible $\uqqn$-module in
$\Oint$ appears as a direct summand of tensor products of\/ $\V$.
\ee
\end{prop}

\subsection{Crystal bases in $\Oint$}

Let $M$ be a $U_q(\mathfrak{q}(n))$-module in
$\Oint$. For $i=1, 2, \ldots, n-1$,
and for a weight vector $u \in M_{\la}$,
consider the  {\em $i$-string
decomposition} of $u$:
$$u =\sum_{k\ge 0} f_i^{(k)} u_k,$$
where $e_i u_k =0$ for all $k \ge 0$, $f_i^{(k)} = f_i^{k} / [k]!$.
We define the {\em even Kashiwara operators} $\tei$,
$\tfi$ $(i=1, \ldots, n-1)$ by
\begin{equation}
\begin{aligned}
& \tei u = \sum_{k \ge 1} f_i^{(k-1)} u_k, \\
& \tfi u = \sum_{k \ge 0} f_i^{(k+1)} u_k.
\end{aligned}
\end{equation}
On the other hand, we define the {\em odd Kashiwara operators}
$\tilde{k}_{\ol {1}}$, $\tilde{e}_{\ol {1}}$, $\tilde{f}_{\ol
{1}}$ by
\begin{equation}
\begin{aligned}
\tkone & = q^{k_1-1}k_{\ol 1}, \\
\teone & = - (e_1 k_{\ol 1} - q k_{\ol 1} e_1) q^{k_1 -1}, \\
\tfone & = - (k_{\ol 1} f_1 - q f_1 k_{\ol 1}) q^{k_2-1}.
\end{aligned}
\end{equation}

Recall that an abstract $\mathfrak{gl}(n)$-crystal is a set $B$
together with the maps
$\tei, \tfi\cl B \to B \sqcup \{0\}$, $\vphi_i,
\eps_i \cl B \to \Z \sqcup \{-\infty\}$ $(i=1, \ldots, n-1)$, and $\wt\cl
B \to P$ satisfying the conditions given in \cite{Kas93}.
We say that
an abstract $\mathfrak{gl}(n)$-crystal is a {\em $\mathfrak{gl}(n)$-crystal}
if it is realized as a crystal basis of a finite-dimensional
integrable $U_q(\mathfrak{gl}(n))$-module.
In particular,
we have
$$\eps_i(b)=\max\{n\in\Z_{\ge0}\,;\,\tei^nb\not=0\}
, \ \vphi_i(b)=\max\{n\in\Z_{\ge0}\,;\,\tfi^nb\not=0\}$$
for any $b$
in a $\mathfrak{gl}(n)$-crystal $B$.

\Def Let $M= \soplus_{\mu \in P^{\ge 0}} M_{\mu}$ be a
$U_q(\mathfrak{q}(n))$-module in the category
$\Oint$. A {\em crystal basis} of $M$ is a
triple $(L, B, l_{B}=(l_{b})_{b\in B})$, where
\bna
\item $L$ is a free $\A$-submodule of $M$ such that

\bni
\item $\F \otimes_{\A} L \isoto M$,

\item $L = \soplus_{\mu \in P^{\ge 0}} L_{\mu}$, where $L_{\mu} = L
\cap M_{\mu}$,

\item  $L$ is stable under the Kashiwara operators $\tei$,
$\tfi$ $(i=1, \ldots, n-1)$, $\tkone$, $\teone$, $\tfone$.
\end{enumerate}

\item $B$ is a finite $\mathfrak{gl}(n)$-crystal together with
the maps $\teone, \tfone \cl B \to B \sqcup \{0\}$ such that

\bni
\item $\wt(\teone b) = \wt(b) + \alpha_1$, $\wt(\tfone b) = \wt(b) -
\alpha_1$,

\item for all $b, b' \in B$, $\tfone b = b'$ if and only if $b = \teone b'$.
\end{enumerate}

\item $l_{B}=(l_{b})_{b \in B}$ is a family of non-zero $\C$-vector spaces
such that

\bni
\item $l_{b} \subset (L/qL)_{\mu}$ for $b \in B_{\mu}$,

\item  $L/qL = \soplus_{b \in B} l_{b}$,

\item $\tkone l_{b} \subset l_{b}$,
\item for $i=1, \ldots, n-1, \ol 1$, we have
\be[{\rm(1)}]
\item
if $\tei b=0$ then $\tei l_{b} =0$, and
otherwise $\tei$ induces an isomorphism
$l_{b}\isoto l_{\tei b}$.
\item
if $\tfi b=0$ then $\tfi l_{b} =0$,
and otherwise $\tfi$ induces an isomorphism $l_{b}\isoto l_{\tfi b}$.
\ee
\end{enumerate}
\end{enumerate}

\edf

As proved in \cite{GJKKK}, for every crystal basis $(L, B, l_{B})$ of a $\uqqn$-module $M$
we have $\teone^2 = \tfone^2 = 0$ as endomorphisms on $L/qL$.

\begin{example}
Let $$\V = \soplus_{j=1}^n \F v_{j} \oplus \soplus_{j=1}^n \F
v_{\ol j}$$ be the vector representation of $U_q(\mathfrak{q}(n))$.
Set $$\mathbf{L} = \soplus_{j=1}^n \A v_{j} \oplus
\soplus_{j=1}^n \A v_{\ol j},$$ $l_{j} = \C v_{j} \oplus \C
v_{\ol j}$, and let $\B$ be the crystal graph given below.

$$
\xymatrix@C=5ex
{*+{\young(1)} \ar@<0.1ex>[r]^-{1}
\ar@{-->}@<-0.9ex>[r]_{\ol 1} & *+{\young(2)} \ar[r]^2 & *+{\young(3)} \ar[r]^3 & \cdots \ar[r]^{n-1} & *+{\young(n)} }
$$
Then $(\mathbf{L}, \B,
l_{\B}=(l_j)_{j=1}^n)$ is a crystal basis of $\V$.
\end{example}

\vskip 2ex

The {\em queer tensor product rule} for the crystal bases of
$\uqqn$-modules in the category $\Oint$ is given by the following
theorem.

\Th \cite[Theorem 2.7]{GJKKK2} \hfill

Let $M_j$ be a $\uqqn$-module in $\Oint$ with crystal basis $(L_j,
B_j, l_{B_j})$ $(j=1,2)$. Set $$B_1 \otimes B_2 = B_1 \times B_2
\quad \text{and} \quad l_{B_{1} \otimes B_{2}}=(l_{b_1} \otimes
l_{b_2})_{b_1 \in B_1, b_2 \in B_2}.$$ Then $$(L_1 \otimes_{\A} L_2,
B_1 \otimes B_2, l_{B_1 \otimes B_2})$$ is a crystal basis of $M_1
\otimes_{\F} M_2$, where the action of the Kashiwara operators on
$B_1 \otimes B_2$ are given as follows.

\begin{equation} \label{eq1:tensor product}
\begin{aligned}
\tei(b_1 \otimes b_2) & = \begin{cases} \tei b_1 \otimes b_2 \ &
\text{if} \ \vphi_i(b_1) \ge \eps_i(b_2), \\
b_1 \otimes \tei b_2 \ & \text{if} \ \vphi_i(b_1) < \eps_i(b_2),
\end{cases} \\
\tfi(b_1 \otimes b_2) & = \begin{cases} \tfi b_1 \otimes b_2 \
& \text{if} \  \vphi_i(b_1) > \eps_i(b_2), \\
b_1 \otimes \tfi b_2 \ & \text{if} \ \vphi_i(b_1) \le \eps_i(b_2),
\end{cases}
\end{aligned}
\end{equation}
\begin{equation} \label{eq2:tensor product}
\begin{aligned}
\teone (b_1 \otimes b_2) & = \begin{cases} \teone b_1 \otimes b_2
& \text{if $\lan k_1, \wt b_2 \ran =\lan k_2, \wt b_2 \ran =0$,} \\
b_1 \otimes \teone b_2
&  \text{otherwise,}
\end{cases} \\
\tfone(b_1 \otimes b_2) & = \begin{cases} \tfone b_1 \otimes b_2
& \text{if $\lan k_1, \wt b_2 \ran = \lan k_2, \wt b_2 \ran =0$,}
 \\
b_1 \otimes \tfone b_2   
& \text{otherwise}.
\end{cases}
\end{aligned}
\end{equation}
 \enth

\Def An {\em abstract $\mathfrak{q}(n)$-crystal} is a
$\mathfrak{gl}(n)$-crystal together with the maps $\teone, \tfone\cl B
\to B \sqcup \{0\}$ satisfying the following conditions:
\bna
\item $\wt(B)\subset P^{\ge0}$,
\item $\wt(\teone b) = \wt(b) + \alpha_1$, $\wt(\tfone b) = \wt(b) -
\alpha_1$,
\item for all $b, b' \in B$, $\tfone b = b'$ if and only if $b = \teone b'$.
\item if $3 \le i \le n-1$, we have
\bni
\item the operators $\te_{\ol{1}}$ and $\tf_{\ol{1}}$ commute $\te_i$ and $\tf_i$.
\item if $\te_{\ol 1} b \in B$, then $\varepsilon_i(\te_{\ol 1} b)=\varepsilon_i( b)$
and $\varphi_i(\te_{\ol 1} b) = \varphi_i(b)$.
\ee
\end{enumerate}
 \edf
For an abstract $\q(n)$-crystal $B$ and an element $b \in B$,
we denote by $C(b)$ the connected component of $b$ in $B$.

Let $B_1$ and $B_2$ be abstract $\qn$-crystals. The {\em tensor
product} $B_1 \otimes B_2$ of $B_1$ and $B_2$ is defined to be the
$\mathfrak{gl}(n)$-crystal $B_1 \otimes B_2$ together with the maps
$\teone$, $\tfone$ defined by \eqref{eq2:tensor product}.
Then it is an abstract $\qn$-crystal. Note that $\otimes$ satisfies
the associativity axiom on the set of abstract $\mathfrak{q}(n)$-crystals.

The next lemma follows directly from \eqref{eq1:tensor product} and \eqref{eq2:tensor product}.
It will be used in Section 2 and Section 4.
\begin{lemma} \label{le_multi_tensor}
Let $B_j$ $(j=1,\ldots, N)$ be abstract $\qn$-crystals, and let $b_j \in B_j$ $(j =1, \ldots, N)$.
\bna
\item For $i \in \{1,,\ldots, n-1, \ol{1}\}$, suppose that
$\tf_i(b_1 \otimes \cdots \otimes b_N) = b_1 \otimes \cdots \otimes  b_{k-1} \otimes \tf_i b_k \otimes b_{k+1} \otimes \cdots \otimes b_N$ for some $1 \leq k \leq N$.
Then for any positive integers $j$ and $m$ such that $1 \leq j \leq k \leq m$, we have
$$\tf_i(b_j \otimes \cdots \otimes b_m) = b_j \otimes \cdots \otimes  b_{k-1} \otimes \tf_i b_k \otimes b_{k+1} \otimes \cdots \otimes b_m.$$

\item For $i \in \{1,,\ldots, n-1, \ol{1}\}$, suppose that $\te_i(b_1 \otimes \cdots \otimes b_N) = b_1 \otimes \cdots \otimes b_{k-1} \otimes \te_i b_k \otimes b_{k+1} \otimes \cdots \otimes b_N$ for some $1 \leq k \leq N$. Then
for any positive integers $j$ and $m$ such that $1 \leq j \leq k \leq m$, we have
$$\te_i(b_j \otimes \cdots \otimes b_m) = b_j \otimes \cdots \otimes b_{k-1} \otimes \te_i b_k \otimes b_{k+1} \otimes \cdots \otimes b_m.$$

\ee
\end{lemma}

\vskip 2ex

\begin{example} \label{ex_abstract_crystal}

\bna
\item If $(L, B, l_{B})$ is a crystal basis of a $\Uq$-module $M$ in the category $\Oint$, then $B$ is an abstract
$\qn$-crystal.

\item The crystal graph $\B$ is an abstract $\qn$-crystal.

\item By the tensor product rule, $\B^{\otimes N}$ is an abstract
$\qn$-crystal. When $n=3$, the $\qn$-crystal structure of $\B
\otimes \B$ is given below.

$$\xymatrix
{*+{\young(1) \otimes \young(1)} \ar[r]^1 \ar@{-->}[d]^{\ol 1} &
 *+{\young(2) \otimes \young(1)} \ar@<-0.5ex>[d]_1 \ar@{-->}@<0.5ex>[d]^{\ol 1} \ar[r]^2&
 *+{\young(3) \otimes \young(1)} \ar@<-0.5ex>[d]_1 \ar@{-->}@<0.5ex>[d]^{\ol 1} \\
 *+{\young(1) \otimes \young(2)} \ar[d]^2 &
 *+{\young(2) \otimes \young(2)} \ar[r]_2 &
 *+{\young(3) \otimes \young(2)} \ar[d]^2 \\
 *+{\young(1) \otimes \young(3)} \ar@<0.5ex>[r]^1 \ar@{-->}@<-0.5ex>[r]_{\ol 1} &
 *+{\young(2) \otimes \young(3)} &
 *+{\young(3) \otimes \young(3)}
 }$$

\end{enumerate}
\end{example}

\vskip 2ex

Let $W$ be the Weyl group of $\gl(n)$ and let $B$ be an abstract
$\qn$-crystal. For $i=1,\ldots n-1$, we define the automorphism
$S_i$ on $B$ by
$$S_i b = \begin{cases}
\tf_i^{\langle h_i, \wt b \rangle} b & \text{if} \ \langle h_i, \wt b \rangle \ge 0, \\
\te_i^{-\langle h_i, \wt b \rangle} b & \text{if} \ \langle h_i, \wt b \rangle \le 0.
\end{cases}$$

As shown in \cite{Kas91}, there exists a unique (well-defined)
action $S : W \to \Aut B$ such that $S_{s_i}=S_i$. Here $s_i$ is the
simple reflection given by $s_i(\la)=\la- \lan h_i, \la \ran
\alpha_i$ $(\la \in \mathfrak{h}^*)$. Note that $\wt(S_w b) = w
(\wt(b))$ for any $w \in W$ and $b \in B$.

For $i=1,\ldots n-1$, set
\eq
&&w_i = s_2 \cdots s_{i} s_1 \cdots s_{i-1}.
\label{def:wi}
\eneq
Then $w_i$ is the shortest element in $W$ such that $w_i(\alpha_i) = \alpha_1$.
We define the {\em odd Kashiwara operators} $\teibar$, $\tfibar$ $(i=2, \ldots, n-1)$ by
$$\teibar = S_{w_i^{-1}} \teone
S_{w_i}, \ \ \tfibar = S_{w_i^{-1}} \tfone S_{w_i}.$$

\begin{definition}
Let $B$ be an abstract $\q(n)$-crystal and $1\le a\le n$. \bna
\item An element $b \in B$ is called a
\emph{$\gl(a)$-highest weight vector}
if $\te_i b = 0$ for $1 \leq i < a$.
\item An element $b \in B$ is called a \emph{$\q (a)$-highest weight vector}
if $\te_i b = \teibar b = 0$ for $1 \leq i < a $.

\item An element $b \in B$ is called a \emph{$\q (n)$-lowest weight
vector} if $S_{w_0}b$ is a $\q (n)$-highest weight vector, where $w_0$ is the longest element of $W$.
\end{enumerate}
\end{definition}

\vskip 2ex

The $\q (n)$-highest (respectively, lowest) weight vectors will be called \emph{highest} (respectively, \emph{lowest) weight vectors}.
We denote by $\mathcal{HW} (\lambda)$  (respectively, $\mathcal{LW} (\lambda)$)
the set of highest (respectively, lowest) weight vectors of weight $\la$ in $\B^{\otimes |\lambda|}$.
The description of $\mathcal{HW} (\lambda)$ (and hence of $\mathcal{LW} (\lambda)$) is given by the following proposition (see Theorem 4.6 (c) in \cite{GJKKK2})

\begin{prop} \label{prop_char.h.w}
An element  $b_0$ in $\B^{\tensor N}$ is a highest weight vector if and only if $b_0 = 1 \tensor \tf_1 \cdots \tf_{j-1} b$ for some $j$
and some highest weight vector $b$ in $\B^{\otimes (N-1)}$ such that
$\wt(b_0) = \wt(b)+\epsilon_j$ is a strict partition.
\end{prop}

The following theorem is part of the main result in \cite{GJKKK2}.

\begin{theorem} \label{th_crystal}
\bna
\item  For any $\la \in \La^{+}$,
there exists a crystal basis
$(L, B, l_{B})$ of the irreducible highest weight module $V(\la)$
such that
\bni
\item $B_{\la} = \{b_{\la} \}$,
\item $B$ is connected.
\end{enumerate}
Moreover, such a crystal basis is unique.
In particular $B$ depends only on $\la$ as an abstract $\q(n)$-crystal.
Hence we may write $B=B(\la)$.

\item The $\qn$-crystal $B(\la)$ has a unique highest weight vector
$b_{\la}$ and unique lowest weight vector $l_{\la}$. 

\end{enumerate}

\end{theorem}

We close this section with preparatory statements that will be useful in the following sections.
\begin{lemma} \label{le_lowest weight vector}
 Let $a \in \B$ and $b \in \B^{\otimes N}$. Then $a \otimes b$ is a lowest weight vector if and only if
 $b$ is a lowest weight vector and $\epsilon_a + \wt b \in w_0 \Lambda^+$.
\end{lemma}

\begin{proof}
Let $a \otimes b \in \B \otimes \B^{\otimes N}$ be a $\gl(n)$-lowest weight vector.
Then $b$ is a $\gl(n)$-lowest weight vector in $\B^{\otimes N}$
and hence $S_{w_0} b$ is the unique $\gl(n)$-highest weight vector
in the $\gl(n)$-connected component containing $b$ in $\B^{\otimes N}$.
By (a) of Lemma 3.3 in \cite{GJKKK2}, it follows that
$S_{w_0}(a \otimes b) = 1 \otimes \tf_1 \cdots \tf_{j-1} S_{w_0} b$ for some $1 \le j \le n$.
Comparing the weights, we have $j=n-a+1$.

Now let $a \otimes b$ be a $\q(n)$-lowest weight vector.
Then $a \otimes b$ is a $\gl(n)$-lowest weight vector and hence
$S_{w_0} (a\otimes b) = 1 \otimes \tf_1 \cdots \tf_{n-a} S_{w_0} b$.
By Proposition \ref{prop_char.h.w}, $S_{w_0} b$ is a $\q(n)$-highest weight vector and $w_0(\epsilon_a + \wt b) \in \Lambda^+$.

Conversely, let $S_{w_0} b$ is a $\q(n)$-highest weight vector and
$w_0(\epsilon_a + \wt b) \in \Lambda^+$. It is straightforward to
check that $a \otimes b$ is a $\gl(n)$-lowest weight vector. Hence
$S_{w_0} (a\otimes b) = 1 \otimes \tf_1 \cdots \tf_{n-a} S_{w_0} b$.
Again, by By Proposition \ref{prop_char.h.w}, we conclude that
$S_{w_0}(a \otimes b)$ is a $\q(n)$-highest weight vector.
\end{proof}

The following corollary immediately follows from the preceding
lemma.
\begin{corollary} \label{cor_lowest}
Let $b_1, \ldots ,b_N$ be elements in $\B$. Then $b_1 \otimes \cdots \otimes b_N$ is a lowest weight vector in $\B^{\otimes N}$ if and only if
$\wt(b_k)+\cdots +\wt(b_N)\in w_0 \Lambda^+$ for all $k =1, \ldots, N$.
\end{corollary}

\begin{definition}
A finite sequence of positive integers $x=x_1 \cdots x_N$ is called
a {\em strict reverse lattice permutation} if for $1 \le k \le N$
and $2 \le i \le n$, the number of occurrences of $i$ is strictly
greater than the number of occurrences of $i-1$ in $x_k  \cdots
x_N$ as soon as $i-1$ appears in  $x_k  \cdots  x_N$.
\end{definition}

Then we can rephrase Corollary~\ref{cor_lowest}
as follows.
\begin{corollary} \label{cor_strict reverse lattice permutation}
A vector $b_1 \otimes \cdots \otimes b_N \in \B^{\otimes N}$ is a
lowest weight vector if and only if it is a strict reverse lattice
permutation.
\end{corollary}
%
%

\vskip 3mm

\section{Semistandard decomposition tableaux}

\subsection{Semistandard decomposition tableaux}
For a strict partition $\lambda = (\lambda_1,\ldots,\lambda_n)$, we
set $|\lambda|:=\lambda_1 +\ldots+\lambda_n$. Recall that $\ell
(\lambda)$ is the number of nonzero $\lambda_i$'s.
\begin{definition}\label{def_ssdt} \hfill
\bna
\item  The \emph{shifted Young diagram of shape $\lambda$}
is an array of square cells in which the $i$-th row has $\lambda_i$
cells, and is shifted $i-1$ units to the right with respect to the
top row. In this case, we  say that $\lambda$ is a \emph{shifted
shape}.
\item A word $u = u_1 \cdots  u_N$  is a \emph{hook word} if there exists $1 \le k \le N$ such that
\begin{equation}\label{hookequation}
u_1 \ge u_2 \ge \cdots \ge u_k <u_{k+1} <\cdots < u_N.
\end{equation}
 Every hook word has the \emph{decreasing part} $u \downarrow   = u_1 \cdots u_k$, and the \emph{increasing part} $u \uparrow = u_{k+1} \cdots u_N$ (note that the decreasing part is always nonempty).
\item A \emph{semistandard decomposition tableau of a shifted shape $\lambda = (\lambda_1, \ldots, \lambda_n)$ }
is a filling $T$ of $\lambda$ with elements of $\{1,2,\ldots,n\}$
such that: \bni
\item the word $v_i$ formed by reading the $i$-th row from left to right is a hook word of length $\lambda_i$,
\item $v_i$ is a hook subword of maximal length in $v_{i+1} v_i$ for $1 \le i \le \ell(\lambda)-1$. 
\end{enumerate}



\item The \emph{reading word} of a semistandard decomposition tableau $T$ is
$$\readw (T) = v_{\ell(\lambda)} v_{\ell(\lambda)-1} \cdots v_1.$$
\end{enumerate}
\end{definition}

\begin{remark} \hfill
\bi
\item
Our definition of a hook word, and hence of a semistandard
decomposition tableau, is different from the one used in
\cite{Serra}, where $u \downarrow  $ is assumed to be strictly
decreasing, while $u \uparrow $ is weakly increasing. Later, we will
consider the $\qn$-crystal structure on the set of all semistandard
decomposition tableaux of a shifted shape $\la$. Then the highest
weight vectors and the lowest weight vectors have simpler forms in
our choice than the ones in \cite{Serra} (see Example \ref{ex_hw}
and Remark \ref{rem_Serrano's SSDT}).

\item The term ``hook word'' in \cite{St} refers to a word $u$ with strictly
decreasing $u \downarrow  $ and strictly increasing $u \uparrow$.
This definition leads to the notion of \emph{standard
decomposition tableaux}.

\item
If there is any, the way to view a word as a semistandard
decomposition tableau is unique. \ee
\end{remark}

We have an alternative criterion to determine
whether a filling of shifted shape $\la$ (equivalently, its reading word) is a semistandard decomposition tableau or not.
\begin{prop} \label{pro_criterion for SSDT}
Let $u= u_1 \cdots u_{\ell} $ and $u'= u'_1 \cdots u'_{\ell'} $ be hook words with $1 \le \ell' < \ell$.
Then $u'u$ is a semistandard decomposition tableau if and only if for $1 \leq i \leq j \leq \ell'$,
\bna
\item if $u_i \leq u'_j$, then $i \neq 1$ and $u'_{i-1} < u'_j$,
\item if $u_i > u'_j$, then $u_i \geq u_{j+1}$.
\ee
 This is equivalent to saying that none of the following conditions
 holds.
\bni
\item $u_1 \leq u'_i$ $(1 \leq i \leq \ell')$,
\item $u'_i \geq u'_j \geq u_{i+1}$ $(i < j \leq \ell')$,
\item $u'_j < u_i < u_{j+1}$ $(i \leq j \leq \ell')$.
\ee

\begin{proof}
Assume that $u'u$ is a semistandard decomposition tableau.

If $u_1 \leq u'_i$ for $1 \leq i \leq \ell'$, then $u'_i u_1 u_2
\cdots u_{\ell}$ is a hook subword of $u'u$ with length $\ell+1$,
which is a contradiction.

If $u'_i \ge u'_j \ge u_{i+1}$ for $i < j \le \ell'$, then $u'_i \in
u' \downarrow $ and hence $u'_1 \cdots u'_i u'_j u_{i+1} \cdots
u_{\ell}$ is a hook subword of $u'u$ of length $\ell +1$, which is a
contradiction.

If $u'_j < u_i < u_{j+1}$ for $i \le j \le \ell'$, then $u_{j+1} \in
u \uparrow $ and hence $u'_1 \cdots u'_j u_i u_{j+1} \cdots
u_{\ell}$ is a hook subword of $u'u$ of length $\ell +1$, which is a
contradiction.

\vskip 1em Now assume that none of (i), (ii), (iii) holds. Suppose
that $v$ is a hook subword of $u'u$ of length greater than $\ell$.
Let $x$ be the first letter in $v \cap u$ and let $x'$ be the last
letter in $v \cap u'$.

\emph{Case 1:}
$x' \ge x$.

Note that $x$ cannot be $u_1$, since (i) does not hold. Let $x$ be
$u_{i+1}$ for some $i \in \{1,2,\ldots, \ell-1\}$ and let $x'$ be
$u'_j$ for some $j \in \{1,2,\ldots, \ell'\}$. Then the length of $v
\cap u$ is less than or equal to $\ell-i$ and hence the length of $v
\cap u'$ is greater than or equal to $i+1$, which implies $i < j$.
Moreover, since $i+1 \ge 2$, $v \cap u'$ contains another letter
besides $u'_j$. Let $u'_k$ be the second last letter in $v \cap u'$
($k < j$). Then we have $k \ge i$, since the length of $v \cap u'$
is greater than or equal to $i+1$. Because $u'_j \ge u_{i+1}$, we
have $u'_j \in v \downarrow $ and hence $u'_k \in v \downarrow$.
Thus we get $u'_k \ge u'_j$. It follows that $u'_k \in u'
\downarrow$ and hence $u'_i \ge u'_k \ge u'_j \geq u_{i+1}$, which
is a contradiction to (ii).

\emph{Case 2:}
$x' < x$.

Let $x$ be $u_i$ for some $i \in \{1,2,\ldots, \ell\}$ and let $x'$
be $u'_j$ for some $j \in \{1,2,\ldots, \ell'\}$. Note that the
length of $v \cap u' $ is less than or equal to $j$ and hence the
length of $v \cap u$ is greater than or equal to $\ell-j+1$, because
the length of $v$ is greater than $\ell$. Moreover we have $i \le
j$, and  $v \cap u$ contains another letter in $u$ besides $u_i$.
Let $u_k$ be the second letter in $v \cap u$ ($k > i$). Note that $k
\le j+1$ since the length of $v \cap u $ is greater than or equal to
$\ell -j +1$. On the other hand, $u'_j < u_i$ implies $u_i, u_k  \in
v \uparrow$. Thus we have $u_i < u_k.$ It follows that $u_k \in u
\uparrow$ so that $u'_j < u_i < u_k \leq u_{j+1}$, which is a
contradiction to (iii).
\end{proof}
\end{prop}

\bigskip

If $T$ is a semistandard decomposition tableau of shifted shape
$\lambda$, we write $\sh(T) = \lambda$. Let $\B (\lambda)$ denote
the set of all semistandard decomposition tableau $T$ with $\sh (T)
= \lambda$. For every $\lambda \in \Lambda^+$, we have the following
embedding
$$\readw: \B (\lambda) \to \B^{\otimes |\lambda|}, \; T \mapsto \readw(T).$$
Using this embedding, we identify $\B(\lambda)$ with a subset in
$\B^{\otimes |\lambda|}$ and define the action of the Kashiwara
operators $\tei, \teibar, \tfi, \tfibar$ on the elements in $ \B
(\lambda)$. The question is whether the set $\B(\lambda)$ is closed
under these operators.

For a strict partition $\lambda$ with $\ell(\lambda) =r$, set
\begin{align}
T^{\lambda}:=&
(1^{\lambda_r}) (2^{\lambda_r} 1^{\lambda_{r-1} - \lambda_r}) \cdots
((r-k+1)^{\lambda_r} (r-k)^{\lambda_{r-1}-\lambda_r} \cdots 1^{\lambda_{k} - \lambda_{k+1}}) \nonumber \\
&\cdots (r^{\lambda_r} (r-1)^{\lambda_{r-1}-\lambda_r} \cdots
1^{\lambda_1 - \lambda_2}), \nonumber
\end{align}
$$L^{\lambda}:=(n-r+1)^{\lambda_r} \cdots (n-k+1)^{\lambda_k} \cdots n^{\lambda_1}.$$
Then we have $S_{w_0} T^{\la} = L^{\la}$.

\bigskip

\begin{example} \label{ex_hw}
Let  $n=4$ and $\lambda = (6,4,2,1)$. Then we have
$$T^{\lambda}=\young(432211,:3211,::21,:::1) \ \text{and} \ L^{\lambda}=\young(444444,:3333,::22,:::1).$$
\end{example}

\bigskip

Our first main result is given in the following theorem.

\begin{theorem} \label{th_SSDT_crystal}
Let $\la$ be a strict partition with $\ell(\lambda)=r$.

\bna
\item The set $ \B (\lambda) \cup \{0\}$ is closed under the action of the Kashiwara operators.
In particular,  $ \B (\lambda) $ becomes an abstract  $\mathfrak{q}(n)$-crystal.

\item The element $T^{\lambda}$ is a unique highest weight vector in $\B (\lambda)$ and
$L^{\lambda}$ is a unique lowest weight vector in $\B (\lambda)$.
\item
The abstract $\q(n)$-crystal $\B(\la)$ is isomorphic to $B(\la)$,
the crystal of the irreducible highest weight module $V(\la)$.
\end{enumerate}

\end{theorem}
\begin{proof}
(a)
{\it Step 1:}
Let $u = u_1 \cdots u_N$ be a hook word such that
$$u_1 \geq u_2 \geq \cdots \geq u_k < u_{k+1} < \cdots < u_N.$$
We will prove that $\tfi u, \tei u$ $(i=1,...,n-1, \ol{1})$ are hook words, when they are nonzero.

Assume that $\tf_i u \neq 0$ for some $i \in \{1,\ldots, n-1 \}$.
Note that $u$ can be regarded as a semistandard tableau of shape $k
\epsilon_1 + \epsilon_2 + \cdots + \epsilon_{N-k+1}$. Since the set
of semistandard tableaux of a skew shape is closed under the action
of the even Kashiwara operators, $\tf_i u$ is a semistandard tableau
of shape $k \epsilon_1 + \epsilon_2 + \cdots + \epsilon_{N-k+1}$ and
hence it is a hook word. For the same reason, we deduce that $\te_i
u$ is a hook word for $i=1, 2, \ldots, n-1$, unless it is zero.

Assume that $\tf_{\ol 1} u \neq 0$. Then we have $u_k =1 $, $u_{j} > 2$ for $j \geq k+1$ and hence
$$\tf_{\ol 1} u = u_1 \cdots u_{k-1} 2 u_{k+1} \cdots u_N.$$
Note that if $u_{k-1} =1$, then $u_1 \geq  \cdots \geq u_{k-1} < 2 < u_{k+1} < \cdots < u_N$,
and if $u_{k-1} \geq 2$ then $u_1 \geq  \cdots \geq u_{k-1} \geq 2 < u_{k+1} < \cdots < u_N$.
In both cases, $\tf_{\ol 1} u $ is a hook word.

Assume that $\te_{\ol 1} u \neq 0$. Since $u_k = \min \{u_j ;
j=1,\ldots N\}$, we have $u_k \leq 2$. If $u_k = 2$, then $u_{j} >
2$ for $j \geq k+1$ and hence $\te_{\ol 1} u = u_1 \cdots u_{k-1} 1
u_{k+1} \cdots u_N$. It follows that $\te_{\ol 1} u$ is a hook word.
If $u_k = 1$, then $u_{k+1}=2$, because $\te_{\ol 1} u \neq 0$.
Hence we get $\te_{\ol 1} u = u_1 \cdots u_{k-1} 1 1 u_{k+2} \cdots
u_n$, which is a hook word.

Let $v_j$ be the reading word of the $j$-th row of a semistandard
decomposition tableau $u$. By Lemma \ref{le_multi_tensor}, we know
that $\tf_i u = v_r \cdots \tf_i v_a \cdots v_1$ for some $1 \le a
\le r$, and $\te_i u = v_r \cdots \te_i v_b \cdots v_1 $ for some $1
\le b \le r$. Hence we conclude all the rows of $\tf_i u$ and $\te_i
u$ are again hook words. \vskip 1em

{\it Step 2:} Let $u = u_1 \cdots u_N$ be a semistandard
decomposition tableau of shifted shape $\lambda$. We will show that
$\tfi u, \tei u$ ($i=1,...,n-1, \ol{1}$) satisfy the condition in
Definition \ref{def_ssdt} (c) (ii), when they are nonzero. We will
prove our claim in four separate cases. \vskip 1em

{\it Case 1:}
For an $i \in \{1,\ldots,n-1\}$, assume that $\tf_i u = u_1 \cdots u_{t-1} u'_t u_{t+1} \cdots u_N $, where $u_t=i$ and $u'_t=\tf_i u_t=i+1$.
Let $v'=u_{j_1} \cdots u_{j_{\ell-1}} u'_t u_{j_{\ell+1}} \cdots u_{j_r}$ be a hook subword of $\tf_i u$.
By Lemma \ref{le_multi_tensor}, taking two consecutive rows of $u$ which contains $u_t$, one can assume that $\lambda_3=0$ from the beginning.
 Then it is enough to show that there exists a hook subword $v$ of $u$ of length $r$.

{\bf (i)} Suppose $u'_t \in v' \downarrow$.
Let $u'_t=u_{j_{\ell+1}}=u_{j_{\ell+2}}= \cdots =u_{j_{\ell+s}}=i+1$ and
$u_{j_{\ell+s+1}} \neq i+1$ for some $s \geq 0$.
Here, we regard $u_{j_{\ell+s+1}}$ as the empty word, if $\ell+s=r$.

If $s=0$, then replacing $u'_t$ by $u_t$ in $v'$, we get a subword
$v=u_{j_1} \cdots u_{j_{\ell-1}} u_t u_{j_{\ell+1}} \cdots u_{j_r}$
of $u$ of length $r$. Since we have
$$\begin{cases}
u_{j_{\ell-1}} > u_t \geq u_{j_{\ell+1}}  & \text{if} \ u_{j_{\ell+1}} < u'_t, \\
u_{j_{\ell-1}} > u_t < u_{j_{\ell+1}} < \cdots < u_{j_r} & \text{if} \ u_{j_{\ell+1}} > u'_t,
\end{cases}$$
$v$ is a hook subword of $u$ of length $r$.

Assume that $s \geq 1$. Since $\tf_i$ acts on $u_t$, we know that
for each $p=1,2,\ldots, s$, there exists $v_p$ between $u_t$ and
$u_{j_{\ell+p}}$ in $u$ such that $v_p=i$. We can assume that $v_p
\neq v_{p'}$ for $p \neq p'$. Replacing $u_{j_{\ell+p}}$ by $v_p$
for each $1 \leq p \leq s$ and $u'_t$ by $u_t$ in $v'$, we obtain a
subword $v=u_{j_1} \cdots u_{j_{\ell-1}}  u_t v_1 \cdots  v_s
u_{j_{\ell+s+1}}  \cdots  u_{j_r}$ of $u$.

If $\ell+s=r$, then we have $v=u_{j_1} \cdots u_{j_{\ell-1}}  u_t v_1 \cdots  v_s$ and it is a hook word.

Assume that $\ell+s < r$. If $u_{j_{\ell+s+1}} \in v' \downarrow$,
then $u_{j_{\ell+s+1}} \leq i$, and hence
$$u_{j_{\ell-1}} > u_t=v_1=\cdots =v_s \geq u_{j_{\ell+s+1}}.$$

If $u_{j_{\ell+s+1}} \in v' \uparrow$, then $u_{j_{\ell+s+1}} > i+1$ and hence
$$u_{j_{\ell-1}} > u_t=v_1=\cdots =v_s < u_{j_{\ell+s+1}} < \cdots < u_{j_r}.$$

In both cases, $v$ is a hook subword of $u$ of length $r$.

{\bf (ii)} Suppose $u'_t \in v' \uparrow$.
Then $j_{\ell-1} \geq 1$ and $u_{j_{\ell-1}} \leq i $.
If $u_{j_{\ell-1}} < i$,  replacing $u'_t$ by $u_t$, we obtain a hook subword
$v=u_{j_1} \cdots u_{j_{\ell-1}} u_t u_{j_{\ell+1}} \cdots u_{j_r}$ of $u$ of length $r$.

If $u_{j_{\ell-1}} = i $, then we know that there exists $u_q$ between $u_{j_{\ell-1}}$ and $u_t$ in $u$
such that $u_q= i+1$.
Replace $u'_t$ by $u_q$ in $v'$.
Then we have a subword $v=u_{j_1} \cdots u_{j_{\ell-1}} u_q u_{j_{\ell +1}} \cdots u_{j_r}$ of $u$ such that
$$u_{j_{\ell-1}}=i < u_q=i+1=u'_t < u_{j_{\ell +1}}.$$
Thus $v$ is a hook subword of $u$ of length $r$.
\vskip 1em

{\it Case 2:}
For an $i \in \{1,\ldots,n-1\}$, assume that $\te_i u = u_1 \cdots u_{t-1} u'_t u_{t+1} \cdots u_N $, where $u_t=i+1$ and $u'_t=\te_i u_t=i$.
Let $v'=u_{j_1} \cdots u_{j_{\ell-1}} u'_t u_{j_{\ell+1}} \cdots u_{j_r}$ be a hook subword of $\te_i u$.
We will show that there exists a hook subword $v$ of $u$ of length $r$.

{\bf (i)} Suppose $u'_t \in v' \downarrow$.
Then we have $u_{j_p} \geq i$ for $p=1, \ldots, j_{\ell-1}$.
Let $u'_t=u_{j_{\ell-1}}=u_{j_{\ell-2}}= \cdots =u_{j_{\ell-s}}=i$ and
$u_{j_{\ell-s-1}} > i$ for some $s \geq 0$.
Here we regard $u_{j_{\ell-s-1}}$ as the empty word, if $\ell-s=1$.

If $s=0$, then replacing $u'_t$ by $u_t$, we obtain a hook subword $v=u_{j_1} \cdots u_{j_{\ell-1}} u_t u_{j_{\ell+1}} \cdots u_{j_r}$ of $u$ of length $r$, since
$$\begin{cases}
u_{j_{\ell-1}} \geq u_t > u_{j_{\ell+1}}  & \text{if} \ u_{j_{\ell+1}} \le i , \\
u_{j_{\ell-1}} \geq u_t = u_{j_{\ell+1}} & \text{if} \ u_{j_{\ell+1}} =i+1, \\
u_{j_{\ell-1}} \geq u_t < u_{j_{\ell+1}} < \cdots < u_{j_r} & \text{if} \ u_{j_{\ell+1}} > i+1.
\end{cases}$$

Assume $s \geq 1$.
Since $\te_i$ acts on $u_t$, we know that for each $p=1,2,\ldots, s$,
there exists $v_p$ between $u_{j_{\ell-p}}$ and $u_t$ in $u$ such that $v_p=i+1$.
We can assume that $v_p \neq v_{p'}$ for $p \neq p'$.
Replace $u_{j_{\ell-p}}$ by $v_p$ for each $0 \leq p \leq s$ and $u'_t$ by $u_t$ in $v'$.
Then we get a subword $v=u_{j_1} \cdots  u_{j_{\ell-s-1}}  v_s \cdots v_1 u_t \cdots  u_{j_{\ell+1}}  \cdots  u_{j_r}$ of $u$ such that
$$u_{j_1} \ge \cdots \ge u_{j_{\ell-s-1}} \geq v_s=\cdots = v_1 = u_t. $$
Since
$$\begin{cases}
 u_t > u_{j_{\ell+1}}  & \text{if} \ u_{j_{\ell+1}} \le i , \\
 u_t = u_{j_{\ell+1}} & \text{if} \ u_{j_{\ell+1}} =i+1, \\
 u_t < u_{j_{\ell+1}} < \cdots < u_{j_r} & \text{if} \ u_{j_{\ell+1}} > i+1,
\end{cases}$$
$v$ is a hook subword of $u$ of length $r$.

{\bf (ii)} Suppose $u'_t \in v' \uparrow$.
If $u_{j_{\ell+1}} = i+1 $, then we know that there exists $u_q$ between $u_t$ and $u_{j_{\ell+1}}$ in $u$
such that $u_q= i$.
Replace $u'_t$ by $u_q$ in $v'$.
Then we have a hook subword $v=u_{j_1} \cdots u_{j_{\ell-1}} u_q u_{j_{\ell +1}} \cdots u_{j_r}$ of $u$.

If $u_{j_{\ell+1}} > i+1 $, then replacing $u'_t$ by $u_t$ in $v'$, we have a word $v=u_{j_1} \cdots u_{j_{\ell-1}} u_t u_{j_{\ell +1}} \cdots u_{j_r}$ such that
$$u_{j_{\ell-1}} < u_t < u_{j_{\ell+1}} < \cdots < u_{j_r}.$$
Hence $v$ is a hook subword of $u$ of length $r$.
\vskip 1em

{\it Case 3:}
Let $\tf_{\ol 1} u = u_1 \cdots u_{t-1} u'_t u_{t+1} \cdots u_N $, where $u_t=1$ and $u'_t=\tf_{\ol 1} u_t=2$.
We have $u_j \geq 3$ for $j \geq t+1$.
Let $v'=u_{j_1} \cdots u_{j_{\ell-1}} u'_t u_{j_{\ell+1}} \cdots u_{j_r}$ be a hook subword of $\tf_{\ol 1} u$ of length $r$.

If $u'_t \in v' \downarrow$, then we have $u_{j_{\ell+1}} \in v' \uparrow$. Replacing $u'_t$ by $u_t$, we obtain a subword $v=u_{j_1} \cdots u_{j_{\ell-1}} u_t u_{j_{\ell+1}} \cdots u_{j_r}$ of $u$ such that
$$u_{j_1} \geq \cdots \geq u_{j_{\ell -1 }} > u_t < u_{j_{\ell+1}} < \cdots < u_{j_r} $$
of length $r$.
It follows that $v$ is a hook subword of $u$.

If $u'_t \in v' \uparrow$, then $1=u_{j_{\ell-1}} \in v' \downarrow$ . Replace $u'_t$ by $u_t$ in $v'$.
Then we obtain a hook subword $v=u_{j_1} \cdots u_{j_{\ell-1}} u_t u_{j_{\ell+1}} \cdots u_{j_r}$ of $u$ of length $r$, since
$$u_{j_1} \geq \cdots \geq u_{j_{\ell -1 }} = u_t < u_{j_{\ell+1}} < \cdots < u_{j_r}.$$
\vskip 1em

{\it Case 4:}
Let $\te_{\ol 1} u = u_1 \cdots u_{t-1} u'_t u_{t+1} \cdots u_N $, where $u_t=2$ and $u'_t=\te_{\ol 1} u_t=1$.
We have $u_j \geq 3$ for $j \geq t+1$.
Let $v'=u_{j_1} \cdots u_{j_{\ell-1}} u'_t u_{j_{\ell+1}} \cdots u_{j_r}$ be a hook subword of $\te_{\ol 1} u$ of length $r$.
Note that $u'_t \in u' \downarrow$.
Replace $u'_t$ by $u_t$ in $v'$,
we obtain a hook subword $v=u_{j_1} \cdots u_{j_{\ell-1}} u_t u_{j_{\ell+1}} \cdots u_{j_r}$ of $u$ such that
$$\begin{cases}
u_{j_1} \geq \cdots \geq u_{j_{\ell-1}} < u_t < u_{j_{\ell +1}} < \cdots < u_{j_r} & \text{if} \ u_{j_{\ell-1}} =1, \\
u_{j_1} \geq \cdots \geq u_{j_{\ell-1}} \geq u_t < u_{j_{\ell +1}} < \cdots < u_{j_r} & \text{if} \ u_{j_{\ell-1}} \geq 2
\end{cases}$$
of length $r$, as desired.
\vskip 1em


(b) It is straightforward to verify that $L^\la$ is a semistandard
decomposition tableau of shape $\la$ and $\wt(L^\la) = w_0 \la$. Let
$\la'=\la-\epsilon_r$. By induction on $|\la|$, we know that
$$L^{\la'}=(n-r+1)^{\la_r-1} (n-r+2) ^{\la_{r-1}} \cdots (n-k+1)
^{\la_{k}} \cdots n^{\la_1}$$ is a unique lowest weight vector in
$\B(\la')$. Note that if $u_1 u_2 \cdots u_N \in \B(\la)$ , then
$u_2 \cdots u_N \in \B(\la')$. Thus we have a crystal embedding
$$\B(\la) \hookrightarrow  \B \otimes \B(\la')$$
given by
$u_1 u_2 \cdots u_N \mapsto u_1 \otimes u_2 \cdots u_N.$

Let $j \otimes L^{\la'}$ be a lowest weight vector in $\B \otimes
\B(\la')$. Then, by Lemma \ref{le_lowest weight vector}, we get
$\la'_{n-j}  > \la'_{n-j+1} +1$ and hence $\la'_{n-j} \geq 2$. There
are three possibilities: (i) $j=n-r$ and $\la'_r \geq 2$, (ii) $j >
n-r+1$, and (iii) $j=n-r+1$.

If $j=n-r$ and $\la'_r \geq 2$, then the word $(n-r)(n-r+1)^{\la'_r}$ formed by the first $(\la'_r +1)$-many letters  of $j \otimes L^{\la'}$ is not a hook word.
Hence $j \otimes L^{\la'} \notin \B(\la)$.

If $j > n-r+1$, then $j \otimes (n-r+2)^{\la_{r-1}}$ is a hook subword of the word formed by the first $(\la_r + \la_{r-1})$-many letters of $j \tensor L^{\la'}$ and
its length is $\la_{r-1} +1$.
It follows that $j \otimes L^{\la'} \notin \B(\la)$ for $j > n-r+1$.

We conclude that $(n-r+1) \otimes L^{\la'} = L^{\la}$ is the only lowest weight vector
in $\B(\la)$. Hence $T^{\la} = S_{w_0} L^{\la}$ is the only highest weight vector in $\B(\la)$.
\vskip 1em

(c) By (a) and (b), the $\qn$-crystal $\B(\lambda)$ is connected,
which proves our assertion.
\end{proof}

\bigskip

\begin{remark} \label{rem_Serrano's SSDT}
One can show that the set of semistandard decomposition tableaux of a shifted shape given in \cite{Serra} (which is different from the one in this paper) also
admits a $\q(n)$-crystal structure by reading a semistandard decomposition tableau from right to left, from top to bottom.
In this setting, the highest weight vector and lowest weight vectors in the set of semistandard decomposition tableaux of
shifted shape $\lambda = (6,4,2,1)$ are given by
$$\overline{T}^{\lambda} = \young(432111,:3212,::21,:::1) \ \text{and} \ \overline{L}^{\lambda}=\young(432134,:4324,::43,:::4).$$ \end{remark}

\bigskip

\begin{example}
\bna
\item Since any word of length 2 is a hook word, we obtain  $\B \otimes \B \simeq \B(2\epsilon_1)$.
Thus the crystal in Example \ref{ex_abstract_crystal} (c) is the $\q(3)$-crystal $\B(2\epsilon_1)$.

\item In Figure~\ref{fi_B(3,1,0)}, we illustrate the $\q(3)$-crystal
$\B(3\epsilon_1 +\epsilon_2)$.
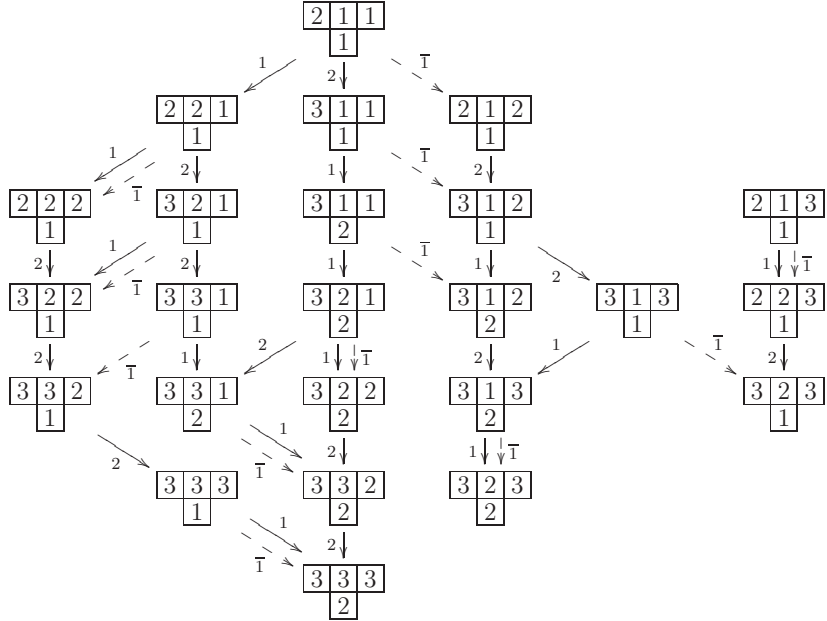
\begin{figure}[!h]
\scalebox{.8}{\xymatrix@R=1pc@H=1pc{ & &  {\young(211,:1)} \ar[dl]_1 \ar_2[d]  \ar^{\ol 1}@{-->}[dr] & & & \\
& {\young(221,:1)} \ar@<-0.5ex>_1[dl] \ar@<1ex>^{\ol 1}@{-->}[dl]
\ar_2[d] & {\young(311,:1)} \ar_1[d]\ar^{\ol 1}@{-->}[dr] &
{\young(212,:1)} \ar_2[d] & & \\
{\young(222,:1)} \ar_2[d] &
{\young(321,:1)}\ar@<-0.5ex>_1 [dl] \ar@<1ex>^{\ol 1}@{-->}[dl] \ar_2[d]
& {\young(311,:2)}  \ar_1 [d] \ar^{\ol 1}@{-->}[dr] & {\young(312,:1)}
\ar_1 [d] \ar_2[dr] & & {\young(213,:1)} \ar@<-0.5ex>_1 [d]
\ar@<1ex>^{\ol 1}@{-->}[d] \\
 {\young(322,:1)} \ar_2[d] & {\young(331,:1)} \ar^{\ol 1}@{-->}[dl] \ar_1 [d] & {\young(321,:2)} \ar@<-0.5ex>_1 [d]
\ar@<1ex>^{\ol 1}@{-->}[d] \ar_2[dl] & {\young(312,:2)}
 \ar_2[d] & {\young(313,:1)} \ar_1 [dl] \ar^{\ol 1}@{-->}[dr] & {\young(223,:1)} \ar_2[d] \\
 {\young(332,:1)} \ar_2[dr] & {\young(331,:2)} \ar@<-0.5ex>@{-->}_{\ol 1}[dr] \ar@<1ex>^{1}[dr]
 &{\young(322,:2)} \ar_2[d] & {\young(313,:2)} \ar@<-0.5ex>_1 [d] \ar@<1ex>^{\ol 1}@{-->}[d] &
 & {\young(323,:1)} \\
   & {\young(333,:1)} \ar@<-0.5ex>@{-->}_{\ol 1}[dr] \ar@<1ex>^{1}[dr]    &  {\young(332,:2)} \ar_2[d] &   {\young(323,:2)} &  &  \\
   & &{ \young(333,:2)}&& &}}
   \caption{$\B(3\epsilon_1 +\epsilon_2)$ for $n=3$.} \label{fi_B(3,1,0)}
   \end{figure}
\ee
   \end{example}

\bigskip

\subsection{The shifted Littlewood-Richardson rule}
We present an explicit combinatorial rule of decomposing the tensor
product of crystal bases of $\uqqn$-modules in the category $\Oint$.
This algorithm is an analogue of the rule of decomposing the tensor
product of crystal bases of $U_q(\gl(n))$-modules in \cite{N93} (see
also \cite{HK2002}), which coincides with the classical
\emph{Littlewood-Richardson rule}.

Let $\la$ be a strict partition. We define $\la \leftarrow j$ to be
the array of cells obtained from the shifted shape $\la$ by adding a
cell at the $j$-th row. Let us denote by $\la \leftarrow j_1
\leftarrow  \cdots  \leftarrow j_r$ the array of cells obtained from
$\la \leftarrow j_1 \leftarrow  \cdots  \leftarrow j_{r-1}$ by
adding a cell at the $j_r$-th row. We define $\B(\la \leftarrow j_1
\leftarrow \cdots \leftarrow j_r)$ to be the null crystal (i.e., the
empty set) unless $\la \leftarrow j_1 \leftarrow \cdots \leftarrow j_k$ is a shifted shape for all $k=1,\ldots,r$.

\begin{theorem}
Let $\la$ and $\mu$ be strict partitions.
Then there is a $\qn$-crystal isomorphism
\begin{align}
\nonumber &\B(\la) \otimes \B(\mu) \simeq \\
\nonumber &\soplus_{u_1 u_2 \cdots u_N \in \B(\la)} \B(\mu \leftarrow(n-u_N+1) \leftarrow (n-u_{N-1}+1) \leftarrow \cdots \leftarrow (n-u_1+1)),
\end{align}
where $N=|\la|$.
\begin{proof}
Let $u_1 \cdots u_N \in \B(\la)$.
The vector $u_1 \cdots u_N \otimes L^{\mu}$ is a lowest weight vector in $\B(\la) \otimes \B(\mu)$
if and only if $w_0 \mu + \epsilon_{u_N} + \epsilon_{u_{N-1}} + \cdots + \epsilon_{u_{k}} \in w_0 \Lambda^+$ for all $k=1, \ldots, N$ by Corollary~\ref{cor_lowest}.
This condition is equivalent to
$$\B(\mu \leftarrow(n-u_N+1) \leftarrow (n-u_{N-1}+1) \leftarrow \cdots \leftarrow (n-u_1+1)) \neq \emptyset.$$
Note that for a lowest weight vector $u_1 \cdots u_N \otimes L^{\mu}$, we have
\begin{align}
C(u_1 \cdots u_N \otimes L^{\mu}) &\simeq B(w_0(\wt(u_1 \cdots u_N ))+\mu) \nonumber \\
& \simeq \B(\mu \leftarrow(n-u_N+1) \leftarrow (n-u_{N-1}+1) \leftarrow \cdots \leftarrow (n-u_1+1)). \nonumber
\end{align}
Thus we have the desired result.
\end{proof}
\end{theorem}

Therefore, we obtain an explicit description of {\it shifted
Littlewood-Richardson coefficients}.
\begin{corollary}
Define
\begin{equation*}
\begin{array}{rl}
\mathcal{LR}_{\la,\mu}^\nu \seteq \{u=u_1\cdots u_N \in \B(\la) \ ; & {\rm (a)} \ \wt(u)= w_0(\nu -\mu) \ \text{and} \\
  {\rm (b)} \ \mu + \epsilon_{n-u_N+1} +  \cdots &+ \epsilon_{n-u_{k}+1} \in  \Lambda^+ \
  \text{for all} \ 1 \leq k \leq N\},
\end{array}
\end{equation*}
and set $f_{\la,\mu}^{\nu} \seteq |\mathcal{LR}_{\la,\mu}^\nu|$.
Then there is a $\qn$-crystal isomorphism
$$\B(\la) \otimes \B(\mu) \simeq \soplus_{\nu \in \Lambda^+} \B(\nu)^{\oplus f_{\la,\mu}^{\nu}}.$$
\end{corollary}

\bigskip

\begin{example} \label{ex_decomposition}
Let $n=3$, $\la=2 \epsilon_1$ and $\mu=3 \epsilon_2 + \epsilon_1$.
For $u_1 u_2 \in \B(\la)$, if $u_2 =1$ then we have
$\mu \leftarrow (3-u_2+1) =\young(\hfill\hfill\hfill,:\hfill,::\hfill)$ so that $\B(\mu \leftarrow (3-u_2+1) \leftarrow ( 3-u_1+1))= \emptyset$.
For the other $u_1 u_2 \in \B(\lambda)$, $\mu \leftarrow (3-u_2 +1) \leftarrow (3-u_1+1)$ is given as follows:
\begin{align}
\young(\hfill\hfill\hfill,:\hfill*,::\bullet) \ (u_1u_2 =12),
&&\young(\hfill\hfill\hfill*,:\hfill,::\bullet) \ (u_1u_2 =13),
&&\young(\hfill\hfill\hfill,:\hfill*\bullet) \ (u_1u_2 =22),  \nonumber \\
\young(\hfill\hfill\hfill*,:\hfill\bullet) \ (u_1u_2 =23),
&&\young(\hfill\hfill\hfill\bullet,:\hfill*) \ (u_1u_2 =32),
&&\young(\hfill\hfill\hfill*\bullet,:\hfill) \ (u_1u_2 =33) \nonumber.
\end{align}
Here, $\young(*)$ and $\young(\bullet)$ denotes the cell added at the first and the second step, respectively.
Hence we have
\begin{align}
\mathcal{LR}_{\la,\mu}^{3 \epsilon_1+2 \epsilon_2+\epsilon_3} = \{ 12 \},
&& \mathcal{LR}_{\la,\mu}^{4 \epsilon_1 + 2 \epsilon_2 } = \{ 23, 32 \},
&& \mathcal{LR}_{\la,\mu}^{5 \epsilon_1 +  \epsilon_2 } = \{ 33 \}. \nonumber
\end{align}
It follows that
$$\B(2 \epsilon_1) \otimes \B(3 \epsilon_1 + \epsilon_2)
\simeq \B(3\epsilon_1 + 2 \epsilon_2 + \epsilon_3) \oplus \B(4 \epsilon_1+2\epsilon_2)^{\oplus 2} \oplus \B(5 \epsilon_1+ \epsilon_2). $$
\end{example}
\vskip 3mm

\section{Insertion scheme}
\subsection{Knuth relation}
Recall that there is an equivalence relation on the set of three
letter words, which is called the \emph{Knuth relation}, on
$\gl(n)$-crystals \cite{BKK}. In this section we introduce  an
equivalence relation on the set of four letter words. It is a
special case of \emph{$\qn$-crystal equivalence}.

\begin{definition}
Let $B_i$ be an abstract $\qn$-crystals and
let $b_i \in B_i$ $(i=1,2)$.
We say that
$b_1$ is \emph{$\qn$-crystal equivalent to} $b_2$ if there exists an isomorphism of crystals
$$B_1 \supseteq C(b_1) \isoto C(b_2) \subseteq B_2,$$
sending $b_1$ to $b_2$.
We denote this equivalence relation by $b_1 \sim b_2$.
\end{definition}

\begin{example}
By Corollary \ref{cor_strict reverse lattice permutation}, we know that $nnnn$, $(n-1)nnn$ and $n(n-1)nn$ exhaust all the lowest weight vectors in $\B^{\otimes 4}$.
Since $\wt((n-1)nnn) = \wt(n(n-1)nn)= w_0(3 \epsilon_1 + \epsilon_2)$,
 we have $C((n-1)nnn) \simeq C(n(n-1)nn) \simeq B(3 \epsilon_1 + \epsilon_2)$, and hence $(n-1)nnn \sim n(n-1)nn$.
This $\qn$-crystal equivalence is a special case of the following proposition.
\end{example}

\begin{prop} [{\rm queer Knuth relation}] \label{prop:Knuth relation}
Let $B_1$ and $B_2$ be the connected components
containing $1121$ and $1211$ in $\B^{\otimes 4}$, respectively.
Then there exists an abstract $\q(n)$-crystal isomorphism $\psi : B_1 \to B_2$ such that
\begin{eqnarray}
\psi(abcd)&
=acbd & \text{if} \ d \leq b \leq a < c \label{eq_Knuth_A} \\
     && \text{or} \ b < d \leq a < c  \label{eq_Knuth_B} \\
     && \text{or} \ b \leq a < d \leq c \label{eq_Knuth_C} \\
     && \text{or} \ a < b < d \leq c, \label{eq_Knuth_D}  \allowdisplaybreaks\\
&=bacd & \text{if} \ b < d \leq c \leq a \label{eq_Knuth_E} \\
     && \text{or} \ d \leq b < c \leq a, \label{eq_Knuth_F}\allowdisplaybreaks\\
&=abdc & \text{if} \ a < d \leq b < c \label{eq_Knuth_G} \\
     && \text{or} \ d \leq a < b < c. \label{eq_Knuth_H}
\end{eqnarray}

\begin{proof}
Let $\B(Y_1) = \{abcd \in \B^{\otimes 4} \, ; \, b< c \ge d\}$ and $\B(Y_2) = \{abcd \in \B^{\otimes 4} \, ; \, a< b \ge c\}$.
Then $\B(Y_1)$(respectively, $\B(Y_2)$) is the set of semistandard tableaux of skew shape $\tiny\young(::\hfill,:\hfill,\hfill\hfill)$(respectively, of skew shape $\tiny\young(::\hfill,:\hfill\hfill,\hfill)$)
and they are abstract $\qn$-crystals(\cite{GJKKK2}).
Because $1121$ is the only highest weight vector in $\B(Y_1)$, we conclude that $B_1 = \B(Y_1)$.
Similarly, $B_2=\B(Y_2)$.
It is straightforward to check that $\psi$ is a bijection between $\B(Y_1)$ and $\B(Y_2)$.

Since $B_1$ and $B_2$ are crystal bases for irreducible highest
weight $\Uq$-modules with highest weight $3 \epsilon_1 +
\epsilon_2$, there exists a unique crystal isomorphism between
them. Because the decomposition of $B(3 \epsilon_1 + \epsilon_2)$
as a $\gl(n)$-crystal is multiplicity free, it is enough to show
that $\psi$ is a $\gl(n)$-crystal isomorphism between $B_1$ and
$B_2$. For a semistandard tableau $T$ and a letter $x$ we denote
$T \leftarrow_{\gl(n)} x$ the tableau obtained by the column
insertion $x$ into $T$. For a word $w=w_1 \cdots w_N$, set $P_{\rm
col}(w) \seteq (\cdots ((w_1 \leftarrow_{\gl(n)} w_2)
\leftarrow_{\gl(n)} w_3 ) \cdots ) \leftarrow_{\gl(n)} w_N $. As
proved in \cite{BKK}, if $P_{\rm col}(w) = P_{\rm col}(w')$ then
$w$ is $\gl(n)$-crystal equivalent to $w'$ (for the definition of
the column insertion scheme and the $\gl(n)$-crystal equivalence,
see \cite{BKK}). Thus it is enough to show that $P_{\rm
col}(abcd)=P_{\rm col}(\psi(abcd))$ for all $abcd \in B_1$. For
example, if $ d \leq b \leq a < c$, then we have $ P_{\rm
col}(abcd)= \ \young(dba,c) = P_{\rm col}(acbd) $. The other cases
can be verified in a similar manner.
\end{proof}
\end{prop}

\vskip 5mm

\subsection{Insertion scheme}
In this section, we present an algorithm of decomposing the tensor
product $\B(\la) \otimes \B(\mu)$, using the \emph{insertion scheme}
for semistandard decomposition tableaux.

\begin{definition} {\rm (cf. \cite{Serra}).}
Let $T$ be a semistandard decomposition tableau of shifted shape $\la$.
For $x \in \B$, we define $T \leftarrow x$ to be a filling of an array of cells
 obtained from $T$ by applying the following procedure:
\bna
\item Let $v_1=u_1 \cdots u_m$ be the reading word of the first row of $T$
such that $u_1 \geq \cdots \geq u_k < \cdots < u_m$ for some $1 \leq k \leq m$.
If $v_1 x$ is a hook word, then put $x$ at the end of the first row and stop the procedure.

\item Assume that $v_1 x$ is not a hook word.
Let $u_j$ be the leftmost element in $v_1 \uparrow$ which is greater
than or equal to $x$. Replace $u_j$ by $x$. Let $u_i$ be the
leftmost element in $v_1 \downarrow$ which is strictly less than
$u_j$. Replace $u_i$ by $u_j$. (Hence $u_i$ is bumped out of the
first row.)

\item Apply the same procedure to the second row with $u_i$ as described in $(a)$ and $(b)$.

\item Repeat the same procedure row by row from top to bottom
until we place a cell at the end of a row of $T$.
\ee
\end{definition}

We identify $T \leftarrow x$ with the word which is obtained by reading each row from left to right and then moving to the next row from bottom to top.

\bigskip

\begin{example}
Since
\begin{align}
\nonumber \young(66135) \leftarrow 2 = \young(66325,:1) &
  \qquad \young(324) \leftarrow 1 = \young(421,:3)
\end{align}
we obtain
\begin{align}
\young(66135,:324) \leftarrow 2 =\young(66325,:421,::3). \nonumber
\end{align}
\end{example}

\vskip 1em In the rest of this section, we will show that $T
\leftarrow x$ is a semistandard decomposition tableau and it is
$\qn$-crystal equivalent to $T \otimes x$. We need the following
lemmas.

\begin{lemma} \label{le_insertion_A}
Let $y_1 < x_1 < \cdots < x_N$ for some $N \geq 1$.
Then for $z \in \B$, we have
\begin{align}
\nonumber (y_1 x_1 \cdots x_N ) z & \sim y_1 x_1 \cdots x_N z && \text{if} \ z > x_N, \\
\nonumber & \sim y_1 x_i x_1 \cdots x_{i-1} z x_{i+1} \cdots x_N & & \text{if} \ x_{i-1} < z \le x_i \ (i \geq 2), \\
\nonumber & \sim y_1 x_1 z x_2 \cdots x_N && \text{if} \ z \le x_1.
\end{align}
\begin{proof}
If $N=1$, it is trivial.

Let $N=2$.
Then we have
\begin{align}
\nonumber (y_1 x_1 x_2 ) z & \sim y_1 x_1 x_2 z && \text{if} \ z > x_2, \\
\nonumber & \sim y_1 x_2 x_1 z && \text{if} \ x_1 < z \le x_2 \ \text{by} \ \eqref{eq_Knuth_D}, \\
\nonumber & \sim y_1 x_1 z x_2&& \text{if} \ z \le x_1 \ \text{by} \ \eqref{eq_Knuth_G} \ \text{or} \ \eqref{eq_Knuth_H}.
\end{align}

Let $N \geq 3$.
If $x_N < z$, it is trivial.
For the case $x_{N-1}  < z \leq x_N$, we have
\begin{align}
\nonumber (y_1 x_1 \cdots x_{N-2} x_{N-1} x_N ) z & \sim y_1 x_1 \cdots (x_{N-2} x_N x_{N-1} z) && \text{by} \ \eqref{eq_Knuth_D} \\
\nonumber & \sim y_1 x_1 \cdots  (x_{N-3} x_N x_{N-2}  x_{N-1}) z && \text{by} \ \eqref{eq_Knuth_D}\\
\nonumber & \qquad \qquad \cdots \\
\nonumber & \sim (y_1 x_N x_1 x_2)\cdots x_{N-1} z && \text{by} \ \eqref{eq_Knuth_D}.
\end{align}

Let $z \leq x_{N-1}$. Then we have
\begin{xalignat}{3}
\nonumber (y_1 x_1 \cdots x_{N-2} x_{N-1} x_N ) z &
\sim y_1 x_1 \cdots (x_{N-2} x_{N-1} z x_N ) & \text{by} \ \eqref{eq_Knuth_G} \ \text{or} \ \eqref{eq_Knuth_H}.
\end{xalignat}
Now our assertion follows from induction on $N$.
 \end{proof}
\end{lemma}

\begin{lemma} \label{le_insertion_B}
Let $y_M \ge y_{M-1} \ge \cdots \ge y_1 < x$ for some $M \geq 1$.
Then for $u \in \B$, we have
\begin{align}
\nonumber (y_M \cdots y_1 x) u  & \sim y_M \cdots y_1 x u && \text{if} \ u > x, \\
\nonumber & \sim y_j y_M \cdots y_{j+1} x y_{j-1} \cdots y_1 u  && \text{if} \ u \le x, \ y_j < x \leq y_{j+1} \ (1 \le j < M ) , \\
\nonumber & \sim y_M x y_{M-1} \cdots y_1 u  && \text{if} \ u \le x, \ y_M < x.
\end{align}
\begin{proof}
If $M=1$, our assertion is trivial.

Suppose $M \ge 2$.
Let $x \ge u$ and $y_j < x $ for some $j \in \{1,2, \ldots, M \}$.
Then we have
\begin{align}
\nonumber (y_j \cdots y_2 y_1 x) u  & \sim y_j \cdots (y_2 x y_1 u)
\nonumber&& \text{by} \ \eqref{eq_Knuth_A} \ \text{or} \ \eqref{eq_Knuth_B} \ \text{or} \ \eqref{eq_Knuth_C} \\
\nonumber& \sim y_j \cdots (y_3 x y_2 y_1) u && \text{by} \ \eqref{eq_Knuth_A} \\
\nonumber&  \qquad \qquad \cdots \\
\nonumber & \sim (y_j x y_{j-1} y_{j-2}) \cdots y_1 u && \text{by} \ \eqref{eq_Knuth_A}.
\end{align}
In particular, we obtain our claim for $j=M$.

If $y_j < x \le y_{j+1}$ ($1 \le j < M $), then we have
\begin{align}
\nonumber (y_M \cdots y_j \cdots y_1 x) u  &
\sim y_M \cdots y_{j+1} (y_j x y_{j-1} y_{j-2}) \cdots y_1 u & \\
\nonumber& \sim y_M \cdots y_{j+2} (y_j y_{j+1} x y_{j-1}) y_{j-2} \cdots y_1 u &&  \text{by} \ \eqref{eq_Knuth_F} \\
\nonumber& \sim y_M \cdots (y_j y_{j+2}  y_{j+1} x ) y_{j-2} \cdots y_1 u &&  \text{by} \ \eqref{eq_Knuth_E} \\
\nonumber& \sim y_M \cdots (y_j y_{j+3}  y_{j+2} y_{j+1} ) x  y_{j-1} \cdots y_1 u &&  \text{by} \ \eqref{eq_Knuth_E} \\
\nonumber&  \qquad \qquad \cdots \\
\nonumber& \sim (y_j y_M y_{M-1} y_{M-2}) \cdots y_{j+1}  x  y_{j-1} \cdots y_1 u &&  \text{by} \ \eqref{eq_Knuth_E}.
\end{align}
\end{proof}
\end{lemma}

\vskip 3mm

\begin{lemma}\label{le_SSDT<-x_A}
Let $\la_1 > \la_2$. The tensor product $\B(\la_2 \epsilon_1) \otimes \B(\la_1 \epsilon_1)$ contains $\B(\la_1 \epsilon_1 + \la_2 \epsilon_2)$
which is the only direct summand isomorphic to $B(\la_1 \epsilon_1 + \la_2 \epsilon_2)$.
\begin{proof}
If $b_1 \otimes b_2$ is a lowest weight vector in $\B(\la_2 \epsilon_1) \otimes \B(\la_1 \epsilon_1)$ of weight
$\la_2 \epsilon_{n-1} + \la_1 \epsilon_n$, then $b_2=  n^{\la_1}$ by Lemma \ref{le_lowest weight vector}.
Since $\wt(b_1) = \la_2 \epsilon_{n-1}$, we get $b_1=(n-1)^{\la_2}$.
Thus we have $b_1 \otimes b_2 = L^{\la_1 \epsilon_1 + \la_2 \epsilon_2}$ and hence $\B(\la_1 \epsilon_1 + \la_2 \epsilon_2) \subseteq \B(\la_2 \epsilon_1) \otimes \B(\la_1 \epsilon_1).$
\end{proof}
\end{lemma}

\begin{lemma}\label{le_SSDT<-x_B}
Let $\la_1 > \la_2$. We have the following $\qn$-crystal decomposition.
$$\B(\la_1 \epsilon_1 + \la_2 \epsilon_2) \otimes \B \simeq
B((\la_1+1)\epsilon_1 + \la_2 \epsilon_2) \oplus B(\la_1 \epsilon_1+ (\la_2+1)\epsilon_2) \oplus B(\la_1 \epsilon_1 + \la_2 \epsilon_2 + \epsilon_3),$$
where the second summand appears if and only if $\la_1 > \la_2 +1, n \ge 2$ and the third summand appears if and only if $\la_2 > 1, n \ge 3$.
The corresponding lowest weight vectors are given as follows:
\bna
\item $(n-1)^{\la_2}n^{\la_1} \otimes n = L^{(\la_1+1) \epsilon_1 + \la_2
\epsilon_2}$,
\item $(n-1)^{\la_2}n^{\la_1-2}(n-1)n \otimes n $ \quad if $\la_1 > \la_2+1, n \geq
2$,
\item $(n-1)^{\la_2-2}(n-2)(n-1)n^{\la_1-2}(n-1)n \otimes n$ \quad if $\la_2 > 1, n\geq 3$.
\ee
In particular, we have
$$\B((\la_1+1)\epsilon_1 + \la_2 \epsilon_2) \subseteq \B(\la_1 \epsilon_1 + \la_2 \epsilon_2) \otimes \B.$$
\begin{proof}
The decomposition follows from $\B \otimes \B(\la_1 \epsilon_1 + \la_2 \epsilon_2) \simeq \B(\la_1 \epsilon_1 + \la_2 \epsilon_2) \otimes \B$ and Theorem 4.6 (c) in \cite{GJKKK2}.

Note that $(n-1)^{\la_2} n^{\la_1} \in \B(\la_1 \epsilon_1 + \la_2 \epsilon_2)$ and
$(n-1)^{\la_2} n^{\la_1} \otimes n = L^{\la_1 \epsilon_1 + \la_2 \epsilon_2}$.
It follows that
$$\B((\la_1+1)\epsilon_1 + \la_2 \epsilon_2) \subseteq \B(\la_1 \epsilon_1 + \la_2 \epsilon_2) \otimes \B.$$

Assume $\la_1 > \la_2 +1$. The assertion for $\la_2=0$ and $\la_1=2$ is trivial. Let $\la_2 > 0$.
One can easily show that $(n-1)^{\la_2} n^{\la_1-2}(n-1)n \in \B(\la_1 \epsilon_1 + \la_2 \epsilon_2)$, using Proposition \ref{pro_criterion for SSDT}.
On the other hand, we have
\begin{align}
\nonumber (n-1)^{\la_2}n^{\la_1-2}(n-1)n n & \sim (n-1)^{\la_2}n^{\la_1-3} (n-1) n n n && \text{by} \ \eqref{eq_Knuth_E} \\
\nonumber & \qquad \qquad \cdots \\
\nonumber &\sim (n-1)^{\la_2+1} n^{\la_1} = L^{\la_1 \epsilon_1 + (\la_2+1) \epsilon_2} && \text{by} \ \eqref{eq_Knuth_E}.
\end{align}
Thus $(n-1)^{\la_2}n^{\la_1-2}(n-1)n \otimes n$ is a unique lowest
weight vector of weight $\la_1 \epsilon_1 + (\la_2+1) \epsilon_2$.

Assume $\la_2 > 1$.
One can show that $(n-1)^{\la_2-2}(n-2)(n-1)n^{\la_1-2}(n-1)n \in \B(\la_1 \epsilon_1 + \la_2 \epsilon_2)$
using Proposition \ref{pro_criterion for SSDT}.
On the other hand, we have
\begin{align}
\nonumber &(n-1)^{\la_2-2}(n-2)(n-1)n^{\la_1-2}(n-1) n n & \\
\nonumber &\sim (n-1)^{\la_2-2}(n-2)(n-1)n^{\la_1-3} (n-1) n n n && \text{by} \ \eqref{eq_Knuth_E} \\
\nonumber &\qquad \qquad \cdots \\
\nonumber &\sim (n-1)^{\la_2-2}(n-2)(n-1)^2 n^{\la_1} && \text{by} \ \eqref{eq_Knuth_E} \\
\nonumber &\qquad \qquad \cdots \\
\nonumber &\sim (n-2)(n-1)^{\la_2} n^{\la_1} = L^{\la_1 \epsilon_1 + \la_2 \epsilon_2 + \epsilon_3} && \text{by} \ \eqref{eq_Knuth_E}.
\end{align}
Thus $(n-1)^{\la_2-2}(n-2)(n-1)n^{\la_1-2}(n-1)n \otimes n$ is a
unique lowest weight vector of weight $\la_1 \epsilon_1 + \la_2
\epsilon_2 + \epsilon_3$.
\end{proof}
\end{lemma}

\begin{lemma}\label{le_SSDT<-x_C}
Let $\la_1 > \la_2 +1$. Then
$\B \otimes \B(\la_2 \epsilon_1) \otimes \B(\la_1\epsilon_1)$
contains $\B(\la_1 \epsilon_1 + (\la_2+1) \epsilon_2)$
which is the only direct summand 
isomorphic to $B(\la_1 \epsilon_1 + (\la_2+1) \epsilon_2)$.
Moreover, we have
$$
 \B(\la_1 \epsilon_1 + (\la_2+1) \epsilon_2) \subseteq \B \otimes \B(\la_1 \epsilon_1 + \la_2 \epsilon_2)
\subseteq \B \otimes \B(\la_2 \epsilon_1) \otimes \B(\la_1\epsilon_1).
$$
\begin{proof}
Let $a \otimes b_1 \otimes b_2 $ be a lowest weight vector in $\B
\otimes \B(\la_2 \epsilon_1) \otimes \B(\la_1 \epsilon_1)$ of weight
$(\la_2+1) \epsilon_{n-1} + \la_1 \epsilon_n$. Then we have $b_2 =
n^{\la_1}$ by Lemma \ref{le_lowest weight vector}. Comparing the
weights, we get $b_1 =(n-1) ^{\la_2}$ and $a=n-1$. Hence $a\otimes
b_1\otimes b_2 = L^{\la_1 \epsilon_1 + (\la_2+1) \epsilon_2}$ and
$b_1 \otimes b_2 = L^{\la_1 \epsilon_1 + \la_2 \epsilon_2}$.
\end{proof}
\end{lemma}

\begin{lemma}\label{le_SSDT<-x_D}
If $\la_1 > \la_2 > 1$, then $\B \otimes \B(\la_2 \epsilon_1)
\otimes \B(\la_1\epsilon_1)$ contains $\B(\la_1 \epsilon_1 +
\la_2\epsilon_2 + \epsilon_3)$ which is the only direct summand 
isomorphic to $B(\la_1 \epsilon_1 + \la_2 \epsilon_2 +
\epsilon_3)$. Moreover, we have
$$
 \B(\la_1 \epsilon_1 + \la_2 \epsilon_2 + \epsilon_3) \subseteq \B \otimes \B(\la_1 \epsilon_1 + \la_2 \epsilon_2)
\subseteq \B \otimes \B(\la_2 \epsilon_1) \otimes \B(\la_1\epsilon_1).
$$
\begin{proof}
Let $a \otimes b_1 \otimes b_2 $ be a lowest weight vector in $\B
\otimes \B(\la_2 \epsilon_1) \otimes \B(\la_1 \epsilon_1)$ of weight
$\epsilon_{n-2} + \la_2 \epsilon_{n-1} + \la_1 \epsilon_n$. Then we
have $b_2 = n^{\la_1}$ by Lemma \ref{le_lowest weight vector}. Hence
$a=n-1$ or $n-2$.

If $a=n-1$, then $b_1= (n-1)^{m_1} (n-2) (n-1)^{m_2}$ for some
nonnegative integers $m_1$ and $m_2$ such that $m_1 +m_2=\la_2-1$.
Since $(n-2)(n-1)^{m_2} \otimes n^{\la_1}$ is a lowest weight vector
by Lemma \ref{le_lowest weight vector},  we have $m_2 > 1$ by
Corollary \ref{cor_strict reverse lattice permutation}. Then
$b_1=(n-1)^{m_1} (n-2) (n-1)^{m_2} $ is not a hook word, which is a
contradiction.

If $a=n-2$, then $b_1=(n-1)^{\la_2}$ and $a \otimes b_1 \otimes b_2
= (n-2) (n-1)^{\la_2} n^{\la_1} = L^{\la_1 \epsilon_1 + \la_2
\epsilon_2 + \epsilon_3}$ and $b_1 \otimes b_2 = L^{\la_1 \epsilon_1
+ \la_2 \epsilon_2}$, as desired.
\end{proof}
\end{lemma}

\vskip 3mm

\begin{lemma} \label{le_SSDT<-x_E}
Let $\la_1 > \la_2 +1$. Then $\B((\la_2+1) \epsilon_1) \otimes \B(\la_1 \epsilon_1)$
 does not have direct summands isomorphic to $B(\la_1 \epsilon_1 + \la_2 \epsilon_2 + \epsilon_3)$.
\begin{proof}
If $\la_2 \leq 1$, there is nothing to prove. Let $\la_2 > 1$. If
$b_1 \otimes b_2$ is a lowest weight vector of weight
$\epsilon_{n-2} + \la_2 \epsilon_{n-1} + \la_1 \epsilon_n$, then
$b_2= n^{\la_1}$, by Lemma \ref{le_lowest weight vector}. Then $b_1=
(n-1)^{m_1} (n-2) (n-1)^{m_2}$ for some nonnegative integers $m_1$
and $m_2$ with $m_1 + m_2 = \la_2$. By Lemma \ref{le_lowest weight
vector},  $(n-2) (n-1)^{m_2} \otimes n^{\la_1}$ is a lowest weight
vector, and hence $m_2 > 1$ by Corollary \ref{cor_strict reverse
lattice permutation}. Then $b_1=(n-1)^{m_1}(n-2)(n-1)^{m_2}$ is not
a hook word, which is a contradiction.
\end{proof}
\end{lemma}

Now we are ready to prove the main result of this section.

\begin{prop} \label{pro_T<-x}
Let $T$ be a semistandard decomposition tableau of shifted shape $\lambda$ and let $x \in \B$.
Then we have
\bna
\item $T \otimes x \sim T \leftarrow x$,
\item $T \leftarrow x$ is a semistandard decomposition tableau of shifted shape $\lambda +\varepsilon_j$ for some $j=1,\ldots,n$.
\ee
\begin{proof}
(a) Let $v_1$ be the reading word of the first row of $T$.
If $v_1x$ is a hook word,
then we have $v_1 \leftarrow x = v_1 \otimes x$ and hence $T \otimes x \sim T \leftarrow x$.

Assume that $v_1 x$ is not a hook word and $v_1 \leftarrow x = y_{j_1} \otimes v'_1 $,
 where $v'_1$ is the hook word of length $\la_1$ obtained from $v_1$ by inserting $x$ into $v_1$ and
 $y_{j_1}$ is the letter bumped out of $v_1$.
Combining Lemma \ref{le_insertion_A} and Lemma \ref{le_insertion_B}, we obtain
$$v_1 \otimes x \sim v_1 \leftarrow x=y_{j_1} \otimes v'_1.$$
Let $v_2$ be the reading word of the second row of $T$.
If $v_2 y_{j_1}$ is a hook word,
then we have
$$v_2 (v_1 \otimes x)  \sim v_2 (v_1 \leftarrow x) = v_2 \otimes (y_{j_1} \otimes v'_1) = (v_2v_1) \leftarrow x,$$
and hence $$T \otimes x \sim T \leftarrow x.$$
If $v_2 y_{j_1}$ is not a hook word, then by inserting $y_{j_1}$ into $v_2$,
we obtain $y_{j_2}$ and $v'_2$ such that $v_2 \leftarrow y_{j_1} \sim y_{j_2} \otimes v'_2$.

Repeating this procedure row by row, we obtain the desired result.
\vskip 1em (b) By the definition of the insertion scheme, it
suffices to show that $b_1 \otimes b_2 \leftarrow x$ is a
semistandard decomposition tableau for any $x \in \B$ and $b_1
\otimes b_2 \in \B(\la_1 \epsilon_1 + \la_2 \epsilon_2)$ with $\la_1
> \la_2$. Note that $\B(\la_1 \epsilon_1 + \la_2 \epsilon_2)
\subseteq \B(\la_2 \epsilon_1) \otimes \B(\la_1 \epsilon_1)$ by
Lemma \ref{le_SSDT<-x_A}. It is straightforward to verify our claim
for $\la_2 =0$. Let $\la_2 > 0$.

If $b_2 \otimes x $ is a hook word, then $b_1 \otimes b_2 \otimes x \in \B((\la_1+1) \epsilon_1 + \la_2 \epsilon_2)$.
Indeed, if there is a hook subword $u$ of $b_1 \otimes b_2 \otimes x$ of length greater than $ \la_1+1$,
then $u$ must contain $x$.
Since $u - \{x\}$ is a hook subword of $b_1 \otimes b_2$ of length greater than $ \la_1$,
we have a contradiction.

Suppose that $b_2 \otimes x$ is not a hook word.
We have $b_2 \otimes x \sim y \otimes b'_2$, where $b'_2$ is the word obtained from $b_2$ by inserting $x$ into $b_2$ and
$y$ is the element bumped out of $b_2$. Note that $b'_2 \in \B(\la_1 \epsilon_1)$.

\vskip 3mm

\emph{Case 1:}
$b_1 \otimes y$ is not a hook word.

We have $b_1 \otimes y \sim z \otimes b'_1$, where $b'_1$ is the word obtained from $b_1$ by inserting $y$ into $b_1$ and
$z$ is the element bumped out of $b_1$.
It follows that $z \otimes b'_1 \otimes b'_2 \sim b_1 \otimes b_2 \otimes x$.
Since $b_2 \otimes x$ is not a hook word, $b_1 \otimes b_2 \otimes x$ does not lie in $\B((\la_1+1)\epsilon_1 + \la_2 \epsilon_2)$ and hence it lies in $B(\la_1\epsilon_1+(\la_2+1)\epsilon_2)$ or in $B(\la_1 \epsilon_1 + \la_2 \epsilon_2 + \epsilon_3)$ in the direct sum decomposition of $\B(\la_1 \epsilon_1 + \la_2 \epsilon_2) \otimes \B$,
 by Lemma \ref{le_SSDT<-x_B}.
Since $z \otimes b'_1 \otimes b'_2 \sim b_1 \otimes b_2 \otimes x$, $z \otimes b'_1 \otimes b'_2$ lies in $B(\la_1\epsilon_1+(\la_2+1)\epsilon_2)$ or in $B(\la_1 \epsilon_1 + \la_2 \epsilon_2 + \epsilon_3)$ in the direct sum decomposition of $\B \otimes \B(\la_2 \epsilon_1) \otimes \B(\la_1 \epsilon_1)$.
By Lemma \ref{le_SSDT<-x_C} and Lemma \ref{le_SSDT<-x_D}, in both cases we conclude that $z \otimes b'_1 \otimes b'_2$ is a semistandard decomposition tableau.

\vskip 3mm

\emph{Case 2:}
$b_1 \otimes y$ is a hook word.

Since $b_2 \otimes x$ is not a hook word, we have $b_1 \otimes b_2 \otimes x$ lies in $B(\la_1 \epsilon_1 + \la_2 \epsilon_2 + \epsilon_3)$ or in $B(\la_1 \epsilon_1 + (\la_2 +1) \epsilon_2)$ in the decomposition of $\B(\la_1 \epsilon_1 + \la_2 \epsilon_2) \otimes \B$, by Lemma \ref{le_SSDT<-x_B}.
Then, from $b_1 \otimes y \otimes b'_2 \sim b_1 \otimes b_2 \otimes x$,
we have $b_1 \otimes y \otimes b'_2 \in B(\la_1 \epsilon_1 + \la_2 \epsilon_2 + \epsilon_3) \oplus B(\la_1 \epsilon_1 + (\la_2 +1) \epsilon_2)$ in the decomposition of $\B((\la_2+1)\epsilon_1) \otimes \B(\la_1 \epsilon_1)$.
By Lemma \ref{le_SSDT<-x_E}, we get $b_1 \otimes y \otimes b'_2 \in B(\la_1 \epsilon_1 + (\la_2 +1) \epsilon_2)$.
Since $\B((\la_2+1)\epsilon_1) \otimes \B(\la_1 \epsilon_1) \subseteq \B \otimes \B(\la_2\epsilon_1) \otimes \B(\la_1 \epsilon_1)$, we conclude that $b_1 \otimes y \otimes b'_2 \in \B(\la_1 \epsilon_1 + (\la_2 +1) \epsilon_2)$ by Lemma \ref{le_SSDT<-x_C}, as desired.
\end{proof}
\end{prop}

Let $T$ and $T'$ be semistandard decomposition tableaux. We define
$T \leftarrow T'$ to be
$$(\cdots((T \leftarrow u_1) \leftarrow u_2) \cdots )\leftarrow u_N,$$
where $u_1 u_2 \cdots u_N$ is the reading word of $T'$.

\begin{corollary} \label{cor_T<-T'}
Let $T$ and $T'$ be semistandard decomposition tableaux of shifted shape $\lambda$ and $\mu$, respectively.
Then $T \leftarrow T'$ is a semistandard decomposition tableau and we have
$$T \otimes T' \sim T \leftarrow T'.$$
\end{corollary}

\begin{proof}
Applying Proposition \ref{pro_T<-x} (b) repeatedly, we conclude that $T \leftarrow T'$ is a semistandard decomposition tableau.
Let $u_1 u_2 \cdots u_N$ be the reading word of $T'$.
Then we have
\begin{eqnarray}
\nonumber T \otimes T' &=& (\cdots((T \otimes u_1) \otimes u_2) \cdots) \otimes u_N \allowdisplaybreaks \\
\nonumber &\sim& (\cdots((T \leftarrow u_1) \otimes u_2) \cdots) \otimes u_N \allowdisplaybreaks\\
\nonumber &\sim& (\cdots((T \leftarrow u_1) \leftarrow u_2) \cdots) \otimes u_N \allowdisplaybreaks\\
\nonumber && \qquad \qquad \cdots \allowdisplaybreaks \\
\nonumber &\sim& (\cdots((T \leftarrow u_1) \leftarrow u_2) \cdots) \leftarrow u_N.\allowdisplaybreaks \\
\nonumber &=& T \leftarrow T'\allowdisplaybreaks.
\end{eqnarray}
\end{proof}

\bigskip

We now give an algorithm of decomposing the tensor product of
$\qn$-crystals using the insertion scheme.
\begin{theorem}
We have the following decomposition of tensor product of $\qn$-crystals.
$$\B(\la) \otimes \B(\mu) \simeq \soplus_{\stackrel{T \in \B(\la) \ ; }{T \leftarrow L^{\mu} = L^{\nu} \ \text{for some} \ \nu \in \Lambda^+}} \B({\rm sh}(T \leftarrow L^{\mu})).$$
\begin{proof}
To decompose $\B(\la) \otimes \B(\mu)$ into a disjoint union of
connected $\qn$-crystals, it is enough to find all the lowest
weight vectors. Let $T \otimes T' \in \B(\la) \otimes \B(\mu)$ be
a lowest weight vector. By Lemma \ref{le_lowest weight vector}, we
know $T' = L^{\mu}$. By Corollary \ref{cor_T<-T'}, $T \otimes
L^{\mu}$ is lowest weight vector if and only if $T \leftarrow
L^{\mu}$ is a lowest weight vector, hence we get the desired
result.
\end{proof}
\end{theorem}

\vskip 3mm

\begin{example} \label{ex_decomposition_2}
Let $n=3$, $\la=2\epsilon_1$ and $\mu=3\epsilon_1+\epsilon_2$.
We have
$$\{T \in \B(2 \epsilon_1) \,;\,T \leftarrow L^{3\epsilon_1 + \epsilon_2} \ \text{is a lowest weight vector} \}
=\{\young(12), \young(23), \young(32), \young(33)\}.$$
Indeed, we obtain
\begin{align}
\young(12) \leftarrow L^{3 \epsilon_1+ \epsilon_2}
&=(((\young(12) \leftarrow 2) \leftarrow 3) \leftarrow 3) \leftarrow 3
=\Big(\Big(\young(22,:1) \leftarrow 3\Big) \leftarrow 3\Big) \leftarrow 3 \nonumber \\
&=\Big(\young(223,:1) \leftarrow 3\Big) \leftarrow 3
=\young(323,:12)  \leftarrow 3   \nonumber \\
&=\young(333,:22,::1) , \nonumber
\end{align}

and similarly we have
\begin{align}
&\young(23) \leftarrow L^{3 \epsilon_1+ \epsilon_2} = L^{4 \epsilon_1 + 2 \epsilon_2}, &&
\young(32) \leftarrow L^{3 \epsilon_1+ \epsilon_2} = L^{4 \epsilon_1 + 2 \epsilon_2}, \nonumber \\
&\young(33) \leftarrow L^{3 \epsilon_1+ \epsilon_2} = L^{5 \epsilon_1 +  \epsilon_2}. \nonumber
\end{align}
For the other vectors in $\B(2\epsilon_1)$, we have
\begin{align}
&\young(11) \leftarrow L^{3 \epsilon_1+ \epsilon_2} = \young(3323,:11),
&&\young(21) \leftarrow L^{3 \epsilon_1+ \epsilon_2} = \young(3323,:21), \nonumber \\
&\young(31) \leftarrow L^{3 \epsilon_1+ \epsilon_2} = \young(3323,:12),
&&\young(22) \leftarrow L^{3 \epsilon_1+ \epsilon_2} = \young(3323,:22), \nonumber \\
&\young(13) \leftarrow L^{3 \epsilon_1+ \epsilon_2} = \young(3333,:12). \nonumber
\end{align}

Hence we conclude
$$\B(2 \epsilon_1) \otimes \B(3 \epsilon_1 + \epsilon_2)
\simeq \B(3\epsilon_1 + 2 \epsilon_2 + \epsilon_1) \oplus \B(4 \epsilon_1+2\epsilon_2)^{\oplus 2} \oplus \B(5 \epsilon_1+ \epsilon_2). $$
\end{example}
\vskip 3mm

\section{The shifted Littlewood-Richardson tableaux}
In this section, we will present two sets of shifted tableaux which
parameterize the connected components in the tensor products of
$\qn$-crystals $\B^{\otimes N}$ and $\B(\la) \otimes \B(\mu)$,
respectively.

Let $\la$ and $\mu$ be  strict partitions with $\mu \subseteq \la$.
A filling of the skew shifted shape $\la / \mu$ is called a \emph{standard shifted tableau of shape $\la / \mu$} if

(a) the entries in each row are strictly increasing from left to right,

(b) the entries in each column are strictly increasing from top to bottom,

(c) it contains each of the letters $1,2,\ldots, |\la / \mu |$ exactly once.

We denote by $\mathcal{ST}(\la / \mu)$ the set of standard shifted tableaux of shape $\la / \mu$.

\subsection{Decomposition of $\B^{\otimes N}$}
There exists a well-known bijection between the set of words with
entries $\{1,2,\ldots,n\}$ and the set of pairs $(P,Q)$, where $P$
is a semistandard Young tableau and $Q$ is a standard Young tableau
of the same shape as $P$. This is called the {\em
Robinson-Schensted-Knuth correspondence}. It can be understood as a
decomposition of the $\gl(n)$-crystal $\B^{\otimes N}$ into a
disjoint union of connected components (see, for example,
\cite{KK01}). Using the insertion scheme presented in the above
section, we can get an analogous decomposition of the
$\q(n)$-crystal $\B^{\otimes N}$.

\begin{definition} \label{def_recording tableaux}
Let $u=u_1 \cdots u_N \in \B^{\otimes N}$.
\bna
\item The {\em insertion tableau $P(u)$ of $u$} is the semistandard decomposition tableau given by
$$P(u)=(\cdots((u_1 \leftarrow u_2) \leftarrow u_3) \cdots) \leftarrow u_N.$$
\item The \emph{recording tableau $Q(u)$ of $u$} is the filling of the shifted shape $\sh(P(u))$ constructed as follows:
\bni
\item the filling $Q(u)$ consists of the cells that are created by the insertion $(\cdots((u_1 \leftarrow u_2) \leftarrow u_3) \cdots) \leftarrow u_N$,
\item if $u_i$ is inserted into $(\cdots((u_1 \leftarrow u_2) \leftarrow u_3) \cdots) \leftarrow u_{i-1}$ to create
a cell at the position $c_i$, then we fill the cell at $c_i$ with the entry $i$.
\ee
\ee
\end{definition}
Note that for any $u \in \B^{\otimes N}$, $Q(u)$ is a standard shifted tableau with the same shape as $P(u)$.

\begin{example}
(a) Let $n=3$ and $u= 2321$. Since
\begin{align}
\nonumber ((2 \leftarrow 3)\leftarrow 2)\leftarrow 1 = (\young(23) \leftarrow 2) \leftarrow 1
= \young(32,:2) \leftarrow 1 = \young(321,:2) \, ,
\end{align}
we have
\begin{align}
\nonumber P(u)=\young(321,:2) \, ,  \qquad Q(u)=\young(124,:3) \, .
\end{align}
(b) Let $n=4$, $\la=(6,4,2,1)$, and $u=1223333444444$. Then we have
\begin{align}
\nonumber P(u)=\young(444444,:3333,::22,:::1) \, ,  \qquad Q(u) = \young(12478\thirteen,:359\twelve,::6\ten,:::\eleven) \, .
\end{align}
\end{example}

Thus  we get  a map
$$\Psi : \B^{\otimes N} \to \bigsqcup_{\stackrel{\la \in \Lambda^+}{\text{with} \ |\la|=N}} \B(\la) \times \mathcal{ST}(\la)$$
given by
$$u=u_1\cdots u_N \mapsto \big( P(u), Q(u) \big).$$

\vskip 3mm

The inverse algorithm $\Psi^{-1}$ of $\Psi$ is given as follows: For
$P \in \B(\la)$ and $Q \in \mathcal{ST}(\la)$, let $Q_k$ be the
standard shifted tableau obtained from $Q$ by removing the cells
with entries $k+1, k+2, \ldots, N$ and let $x_k$ be the letter in
$P$ at the cell in the same position as $Q_k - Q_{k-1}$ for each
$k$.

\bna
\item If $x_N$ lies in the first row of $P$, then set $u_N \seteq x_N$.
\item Suppose that $x_N$ lies in the $\ell$-th row of $P$($\ell \ge 2$).
Let $v=y_1 \cdots y_{\la_{\ell-1}}$ be the reading word of
$(\ell-1$)-th row of $P$. Suppose that $$y_1 \geq \cdots \geq y_k <
\cdots < y_{\la_{\ell-1}}.$$ If $k=1$ or $x_N \geq y_1$, then $x_N
v$ is a hook word of length $\la_{\ell-1} +1$. Hence we have $k > 1$
and $x_N < y_1$. Let $y_i$ be the rightmost element in $y_1 \cdots
y_{k-1}$ which is strictly greater than $x_N$. Replace $y_i$ by
$x_N$. Let $y_j$ be the rightmost element in $y_k \cdots
y_{\la_{\ell-1}}$ which is less than or equal to $y_i$. Replace
$y_j$ by $y_i$. (Hence $y_j$ gets bumped out of $v$.)
\item Apply the same procedure to the $(\ell -2)$-th row of $P$ with $y_j$ as described in (a) and (b).
\item Repeat the same procedure row by row from bottom to top until an element, say $u_N$, gets bumped out of the first row.
\item Let $P_{N-1}$ be the filling of the array of shape $\sh(Q_{N-1})$ obtained from $P$ by applying (a), (b), (c), and (d).
    Repeat (a), (b), (c), and (d) with $T_{N-1}$ and $Q_{N-1}$ so that an element, say $u_{N-1}$, gets bumped out of $P_{N-1}$.
\item Repeat the whole procedure until we get $N$-many letters $u_N, \ldots, u_1$ to get a word $u=u_1 \cdots u_N \in \B^{\otimes N}$.
\ee

As a consequence, $\Psi$ is a bijection between $\B^{\otimes N}$  and $\displaystyle \bigsqcup_{\stackrel{\la \in \Lambda^+}{\text{with} \ |\la|=N}} \B(\la) \times \mathcal{ST}(\la)$.

\begin{lemma} \label{le_T_sim_T'}
Let $T \in \B(\la)$ and $T' \in \B(\mu)$ for some strict partitions $\la$ and $\mu$.
If $T \sim T'$, then
$\la=\mu$ and $T=T'$.
\begin{proof}
Let $T^{\la} = \te_{i_1}^{a_1} \tf_{i_1}^{b_1} \cdots \te_{i_r}^{a_r} \tf_{i_r}^{b_r} T$
for some $i_1, \ldots i_r \in \{1,\ldots, n-1, \ol{1} \}$, and $a_1,\ldots, a_r$, $b_1, \ldots , b_r$ $\in \Z_{\ge 0}$.
Since $T \sim T'$, $\te_{i_1}^{a_1} \tf_{i_1}^{b_1} \cdots \te_{i_r}^{a_r} \tf_{i_r}^{b_r} T'$ is a highest weight vector in $\B(\mu)$. It must be $T^{\mu}$, since $\B(\mu)$ has a unique highest weight vector.
Because $\wt(T) = \wt(T')$, we get $\la=\mu$. It follows that $T^{\la} = T^{\mu}$ and hence $T=T'$.
\end{proof}
\end{lemma}

\vskip 3mm

\begin{lemma} \label{le_recording_tableaux}
 Let $u=u_1 \cdots u_N \in \B^{\otimes N}$ and let $i \in \{1, \ldots, n-1, \ol{1}\}$.
 If $\tf_i u \neq 0$, then $Q(\tf_i u ) = Q(u)$ and
 if $\te_i u \neq 0$, then $Q(\te_i u ) = Q(u)$.
\begin{proof}
Let $\tf_i u = u'_1 \cdots u'_N$ such that $u'_j =u_j$ for $ j \neq k$ and $i+1 = u'_k =\tf_i u_k$ for some $1 \le k \le N$.
It is enough to show that $\sh((\cdots((u_1 \leftarrow u_2)  \cdots) \leftarrow u_t)) = \sh((\cdots((u'_1 \leftarrow u'_2)  \cdots) \leftarrow u'_t))$ for all $1 \le t \le N$.

If $1 \le t < k$, our assertion is trivial.

If $ t \ge k $, then
\begin{eqnarray*}
(\cdots((u'_1 \leftarrow u'_2)  \cdots) \leftarrow u'_t) &\sim& u'_1 \otimes \cdots \otimes u'_t \\
&=& \tf_i(u_1 \otimes \cdots \otimes u_t ) \quad \text{by} \ \text{Lemma} \ \ref{le_multi_tensor} \\
&\sim& \tf_i (\cdots((u_1 \leftarrow u_2)  \cdots) \leftarrow u_t).
\end{eqnarray*}
By Lemma \ref{le_T_sim_T'}, we have
$$(\cdots((u'_1 \leftarrow u'_2)  \cdots) \leftarrow u'_t) =
\tf_i (\cdots((u_1 \leftarrow u_2)  \cdots) \leftarrow u_t).$$
Hence, by Theorem \ref{th_SSDT_crystal}, we have
\begin{eqnarray*}
\sh((\cdots((u'_1 \leftarrow u'_2)  \cdots) \leftarrow u'_t))
&=&\sh(\tf_i (\cdots((u_1 \leftarrow u_2)  \cdots) \leftarrow u_t)) \\
&=&\sh((\cdots((u_1 \leftarrow u_2)  \cdots) \leftarrow u_t)),
\end{eqnarray*}
as desired.

The proof for $\te_i u$ is similar.
\end{proof}
\end{lemma}

\vskip 3mm

For a standard shifted tableau $Q$ with $|\sh(Q)|=N$,
we define
$$\B_Q = \{ u_1 \cdots u_N \in \B^{\otimes N} \, ; \, Q(u)=Q\}.$$
Now we prove one of the main results of this section.

\begin{theorem} We have the following decomposition of the $\qn$-crystal $\B^{\otimes N}$ into a disjoint union of connected components:
$$ \B^{\otimes N} = \soplus_{\stackrel{\la \in \Lambda^+ }{ \text{with} \ |\la|=N}} \Big(\soplus_{Q \in \mathcal{ST}(\la)} \B_Q \Big),$$
where $\B_Q$ is isomorphic to $\B(\la)$ with $sh(Q)=\la$.
\begin{proof}
As a set, we have
$$ \B^{\otimes N} = \bigsqcup_{\stackrel{\la \in \Lambda^+ }{ \text{with} \ |\la|=N}} \Big(\bigsqcup_{Q \in \mathcal{ST}(\la)} \B_Q \Big).$$
Let $Q$ be a standard shifted tableau of shape $\la \in \La^+$.
By Lemma \ref{le_recording_tableaux}, $\B_Q \cup \{0\}$ is closed under the Kashiwara operators.
Since $\Psi$ is a bijection, we have
$$\B_Q = \{ \Psi^{-1}(T,Q) \,;\, T \in \B(\la)\}.$$
It follows that the map $P : \B_Q \to \B(\la)$ given by $u \mapsto P(u)$ is a bijection.

For any word $u \in \B^{\otimes N}$ and $i \in \{1, \ldots, n-1, \ol{1}\}$,
we know
$$P(\tf_i u) \sim \tf_i u \sim \tf_i P(u).$$
By Lemma \ref{le_T_sim_T'}, we get $P(\tf_i u) = \tf_i P(u)$.
Similarly we have $P(\te_i u) = \te_i P(u)$.
Hence $P : \B_Q \to \B(\la)$ is a $\qn$-crystal isomorphism.
\end{proof}
\end{theorem}

As an immediate consequence, we obtain the following corollary.
\begin{corollary}
For a strict partition $\la$ with $|\la|=N$, let $f^{\la}$ be the number of standard shifted tableaux of shape $\la$.
Then we have
$$\B^{\otimes N} \simeq \soplus_{\stackrel{\la \in \Lambda^{+}}{\text{with} \ |\la|=N}} \B(\la)^{\oplus f^{\la}}.$$
\end{corollary}

\vskip 3mm

\subsection{The shifted Littlewood-Richardson tableaux}
In this section, we will define the notion of {\it shifted
Littlewood-Richardson tableaux of skew shape}, which  parameterize
the connected components of the tensor product $\B(\la) \otimes
\B(\mu)$.

For semistandard decomposition tableaux $T$ and $T'$, we define $T \rightarrow T'$  by
$$u_1 \leftarrow (u_2 \leftarrow \cdots \leftarrow (u_{N-1}\leftarrow (u_N \leftarrow T')) \cdots ),$$
where $u_1 u_2 \cdots u_N$ is the reading word of $T$.
\begin{definition}
Let $T \in \B(\la)$ and $T' \in \B(\mu)$ for some strict partitions $\la, \mu$
and let $\sh(T \rightarrow T') = \nu$.
The \emph{recording tableau $Q(T \rightarrow T')$ of the insertion $T \rightarrow T'$} is the filling of the skew shifted shape $\nu / \mu$ constructed as follows:
\bni
\item the filling $Q(T \rightarrow T')$ consists of the cells that are created by the insertion $T \rightarrow T'$,
\item if $u_{N-k+2} \leftarrow (\cdots  \leftarrow (u_{N-1}\leftarrow (u_N \leftarrow T')) \cdots )$ is inserted into $u_{N-k+1}$  to create a cell at the position $c_k$,
 then we fill the cell at $c_k$ with the entry $k$.
\ee
\end{definition}

\vskip 3mm

\begin{example}
(a) Let $T= \young(12)$ and $T'=\young(333,:2)$.
Then we have
\begin{align}
\nonumber T \rightarrow T'&=\young(1) \leftarrow \Big(\young(2)  \leftarrow \young(333,:2)\Big)
=\young(1) \leftarrow \young(333,:22)
=\young(333,:22,::1) \, .
\end{align}
Hence the recording tableau $Q(T \rightarrow T')$ is given by
$$\young(\hfill\hfill\hfill,:\hfill1,::2) \, .$$

(b) Let $T =\young(312,:2)$  and $T' = \young(322,:1)$.
Then we have
\begin{align*}
T \rightarrow T' = \young(3322,:221,::1) \, , & \qquad Q(T \rightarrow T') = \young(\hfill\hfill\hfill3,:\hfill14,::2) \, .
\end{align*}
\end{example}

\begin{remark}
\bna
\item For any insertion $T \rightarrow T'$, the recording tableau $Q(T \rightarrow T')$ is a standard shifted tableau of shape $\nu / \mu$, where $\sh(T') =\mu$ and $\sh(T \rightarrow T')=\nu$.
\item Since
\begin{align*}
       & u_1 \leftarrow (u_2 \leftarrow \cdots \leftarrow (u_{N-1}\leftarrow (u_N \leftarrow T')) \cdots ) \\
\sim & u_1 \otimes (u_2 \leftarrow \cdots \leftarrow (u_{N-1}\leftarrow (u_N \leftarrow T')) \cdots ) \\
     & \qquad \cdots  \\
\sim & u_1 \otimes (u_2 \otimes  \cdots \otimes (u_{N-1}\otimes (u_N \leftarrow T')) \cdots ) \\
\sim & u_1 \otimes (u_2 \otimes  \cdots \otimes (u_{N-1}\otimes (u_N \otimes T')) \cdots ),
\end{align*}
we have
$T \rightarrow T' \sim T \otimes T' \sim T \leftarrow T', $
and  $T \rightarrow T' = T \leftarrow T'$ by Lemma \ref{le_T_sim_T'}.
\ee
\end{remark}
\begin{lemma} \label{le_row}
Let $T \in \B(\la)$ and $T' \in \B(\mu)$ for some strict partitions
$\la , \mu$. If $T \rightarrow L^{\mu} = L^{\nu}$ for some strict
partition $\nu$, then the reading word of $T$ is given by
\begin{equation}
(n-r_{|\la|}+1) \otimes (n-r_{|\la|-1} +1) \otimes \cdots \otimes (n-r_1+1),
\end{equation}
where $r_k$ denotes the row of the entry $k$ in $Q(T \rightarrow T')$.
\begin{proof}
Let $|\la|=N$ and let $u_1 \cdots u_N$ be the reading word of $T$.
Since $$L^{\nu} = T \rightarrow L^{\mu} \, \sim u_1 \otimes \cdots
\otimes u_N \otimes L^{\mu},$$ we see that $u_k \otimes \cdots
\otimes u_N \otimes L^{\mu}$ is a lowest weight vector for all $k =
1,2, \ldots ,N$, by Corollary \ref{cor_lowest}. Then we have
$$u_k \leftarrow (u_{k+1} \leftarrow \cdots ( u_{N-1} \leftarrow ( u_N \leftarrow L^{\mu})) \cdots)
= L^{\mu+\epsilon_{n-u_N+1} + \cdots +\epsilon_{n-u_k +1}}$$
for all $k =1,2, \ldots, N$, by Lemma \ref{le_T_sim_T'}.
Hence the cell created by inserting $(u_{k+1} \leftarrow \cdots (u_N \leftarrow L^{\mu})\cdots)$
into $u_k$ lies at the $(n-u_k+1)$-th row of $L^{\mu+\epsilon_{n-u_N+1} + \cdots + \epsilon_{n-u_k +1}}$.
Thus we have $u_k = n- r_{N-k+1} +1$ for all $k = 1,2, \ldots, N$, as desired.
\end{proof}
\end{lemma}

For a standard shifted tableau $Q$ of shape $\nu / \mu$, let
$$\B(\la, \mu)_Q = \set{T \otimes T' \in \B(\la) \otimes \B(\mu)}{Q(T \rightarrow T')=Q}.$$

\begin{prop} \label{pro_recording tableaux2}
For strict partitions $\la, \mu, \nu$ with $\mu \subseteq \nu$, let
$Q$ be a standard shifted tableau of shape $\nu / \mu$. \bna
\item The set $\B(\la,\mu)_Q \cup \{0\}$ is closed under the action of the Kashiwara operators.
\item The $\q(n)$-crystal $\B(\la,\mu)_Q$ is isomorphic to $\B(\nu)$.
\ee
\begin{proof}
(a) Let $T \in \B(\la), \ T' \in \B(\mu)$ and let $u_1 \cdots u_N$
be the reading word of $T$. Fix an $i \in \{1,2,\ldots, n-1,
\ol{1}\}$. We will only show that $\B(\la,\mu)_Q \cup \{0\}$ is
closed under $\tf_i$ since the proof for $\te_i$ is similar.

 Suppose that $\tf_i (T \otimes T') = \tf_i T \otimes T'$.
 Then there exists $k \in \{1,2,\ldots, N \}$ such that $$\tf_i(T \otimes T') = u_1 \otimes \cdots \otimes \tf_i u_k \otimes \cdots \otimes u_N \otimes T'.$$
 Observe that for all $j \le k$,
\begin{align}
& u_j \leftarrow ( u_{j+1} \leftarrow \cdots (\tf_i u_k \leftarrow (u_{k+1} \leftarrow \cdots (u_N \leftarrow T')\cdots ))\cdots ) &\nonumber \\
\sim & u_j \otimes u_{j+1} \otimes \cdots \otimes \tf_i u_k \otimes u_{k+1} \otimes \cdots \otimes u_N \otimes T'& \nonumber \\
=& \tf_i(u_j \otimes u_{j+1} \otimes \cdots \otimes  u_k \otimes u_{k+1} \otimes \cdots \otimes u_N \otimes T') && \text{by Lemma} \
\ref{le_multi_tensor} \nonumber \\
\sim & \tf_i(u_j \leftarrow ( u_{j+1} \leftarrow \cdots (u_k \leftarrow (u_{k-1} \leftarrow \cdots (u_N \leftarrow T')\cdots ))\cdots )). \nonumber
\end{align}
Thus, by Lemma \ref{le_T_sim_T'}, we have
\begin{align}
&\tf_i(u_j \leftarrow ( u_{j+1} \leftarrow \cdots (u_k \leftarrow (u_{k+1} \leftarrow \cdots (u_N \leftarrow T')\cdots ))\cdots )) \nonumber \\
=&u_j \leftarrow ( u_{j+1} \leftarrow \cdots (\tf_i u_k \leftarrow (u_{k+1} \leftarrow \cdots (u_N \leftarrow T')\cdots ))\cdots ). \nonumber
\end{align}
It follows that
\begin{align}
&\sh(u_j \leftarrow ( u_{j+1} \leftarrow \cdots (u_k \leftarrow (u_{k+1} \leftarrow \cdots (u_N \leftarrow T')\cdots ))\cdots )) \nonumber \\
=&\sh(u_j \leftarrow ( u_{j+1} \leftarrow \cdots (\tf_i u_k \leftarrow (u_{k+1} \leftarrow \cdots (u_N \leftarrow T')\cdots ))\cdots )) \nonumber
\end{align}
for all $j \leq k$.
Hence we get $Q(T \rightarrow T') = Q(\tf_i T \rightarrow T')$.

By a similar argument, one can show that if $\tf_i(T \otimes T') = T \otimes \tf_i T'$ then $Q(T \rightarrow T') = Q( T \rightarrow \tf_i T')$.
\vskip 3mm
(b) If $T \otimes T'$ is a lowest weight vector in the $\q(n)$-crystal $\B(\la,\mu)_Q$, then we have
$T' = L^{\mu}$ and $T \rightarrow T' =L^{\nu}$.
By Lemma \ref{le_row}, we get
$$T=(n-r_{|\la|}+1) \otimes (n-r_{|\la|-1} +1) \otimes \cdots \otimes (n-r_1+1),$$
where $r_k$ denotes the row of the entry $k$ in $Q(T \rightarrow T')$.
So the $\q(n)$-crystal $\B(\la, \mu)_Q$ in $\B^{\otimes (|\la|+|\mu|)}$
has a unique lowest weight vector $T \otimes T'$ which is $\q(n)$-crystal equivalent to $L^{\nu}$.
It follows that $\B(\la, \mu)_Q \simeq \B(\nu)$, as desired.
\end{proof}
\end{prop}

Now we define the notion of {\em shifted Littlewood-Richardson
tableaux}.
\begin{definition}
Let $\la, \mu, \nu$ be strict partitions with $\mu \subseteq \nu$.
We define
\begin{equation*}
\begin{array}{rl}
  \mathcal{\widetilde{LR}}_{\la,\mu}^\nu  \seteq & \{Q \in \mathcal{ST}(\nu / \mu)  ;
  (n-r_{|\la|}+1) 
  \otimes \cdots \otimes (n-r_1+1) \in \B(\la), \\
  &\text{where $r_k$ denotes the row of the entry $k$ in $Q$}\}.
\end{array}
\end{equation*}
A tableau $Q$ in $\mathcal{\widetilde{LR}}_{\la,\mu}^\nu$ is called a {\em shifted Littlewood-Richardson tableau of shape $\nu / \mu$ and type $\la$}.
\end{definition}
\vskip 3mm

Let $T_1 \otimes T_2 \in \B(\la) \otimes \B(\mu)$ and $\sh(T_1 \rightarrow T_2) = \nu$ for some strict partition $\nu$.
Since $T_1 \otimes T_2 \sim T_1 \rightarrow T_2$,
there exists a $\q(n)$-crystal isomorphism
$$\psi : C(T_1 \otimes T_2) \to \B(\nu)$$
sending $T_1 \otimes T_2$ to $T_1 \rightarrow T_2$.
Let $T \otimes L^{\mu} = \psi^{-1}(L^{\nu})$ for some $T \in \B(\la)$.
We obtain $T \rightarrow L^{\mu}=L^{\nu}$.
By Proposition \ref{pro_recording tableaux2} (a), we have $Q(T_1 \rightarrow T_2) = Q(T \rightarrow L^{\mu})$.
Note that $Q(T \rightarrow L^{\mu}) \in \mathcal{\widetilde{LR}}_{\la,\mu}^\nu$, by Lemma \ref{le_row}.
Hence we have $Q(T_1 \rightarrow T_2) \in \mathcal{\widetilde{LR}}_{\la,\mu}^\nu$.

Conversely, if $Q \in \mathcal{\widetilde{LR}}_{\la,\mu}^\nu$,
then there exists an element $T_1 \otimes T_2 \in \B(\la) \otimes
\B(\mu)$ such that $Q(T_1 \rightarrow T_2) = Q$. For example, if
$T_1=(n-r_{|\la|}+1) \otimes (n-r_{|\la|-1} +1) \otimes \cdots
\otimes (n-r_1+1)$ and $T_2=L^{\mu}$, then we have $Q(T_1
\rightarrow T_2) = Q$, where $r_k$ denotes the row of the entry
$k$ in $Q$.

To summarize, we obtain another main result of this section.

\begin{theorem} Let $\la, \mu$ be strict partitions. Then there exists a
decomposition of the tensor product of $\q(n)$-crystals as follows:
$$\B(\la) \otimes \B(\mu) = \soplus_{\stackrel{\nu \in \Lambda^+ }{\text{with} \ \mu \subseteq \nu}}  \soplus_{Q \in \mathcal{\widetilde{LR}}_{\la,\mu}^\nu} \B(\la, \mu)_Q.$$
Consequently, we have $| \mathcal{\widetilde{LR}}_{\la,\mu}^{\nu} |
=f_{\la, \mu}^{\nu}$.
\end{theorem}

\vskip 3mm

\begin{example}
\bna
\item For $n=3$, let $\la=2 \epsilon_1$, $\mu= 3 \epsilon_1 + \epsilon_2$,
 $\nu_1 = 3 \epsilon_1 + 2 \epsilon_2 + \epsilon_3$, $\nu_2 = 4 \epsilon_1 + 2 \epsilon_2$ and $\nu_3 = 5 \epsilon_1 + \epsilon_2$.
For each $i=1,2,3$ the set $\mathcal{\widetilde{LR}}_{\la,\mu}^{\nu_i}$ is the same as the set
$\mathcal{ST}(\nu_i / \mu)$,
since any word of length 2 is a semistandard decomposition tableau of shifted shape $\la$.
So we have
 \begin{align*}
 & \ \mathcal{\widetilde{LR}}_{\la,\mu}^{\nu_1}=\Bigg\{\young(\hfill\hfill\hfill,:\hfill1,::2)
 \Bigg\},\\
& \
\mathcal{\widetilde{LR}}_{\la,\mu}^{\nu_2}=\Bigg\{\young(\hfill\hfill\hfill1,:\hfill2),
\young(\hfill\hfill\hfill2,:\hfill1)\Bigg\}, \\
&  \
\mathcal{\widetilde{LR}}_{\la,\mu}^{\nu_3}=\Bigg\{\young(\hfill\hfill\hfill12,:\hfill)\Bigg\}.
 \end{align*}
Hence we obtain the same decomposition as in Example \ref{ex_decomposition}, and Example \ref{ex_decomposition_2}.

\item For $n=3$, let $\la=\mu=3 \epsilon_1 + \epsilon_2$ and $\nu= 4\epsilon_1 + 3\epsilon_2 + \epsilon_3$.
Then we have
$$\mathcal{ST}(\nu / \mu) = \Bigg\{ \young(\hfill\hfill\hfill1,:\hfill23,::4)\, ,
  \ \young(\hfill\hfill\hfill1,:\hfill24,::3)\, ,
  \ \young(\hfill\hfill\hfill2,:\hfill13,::4)\, ,
  \ \young(\hfill\hfill\hfill2,:\hfill14,::3)\, ,
  \ \young(\hfill\hfill\hfill3,:\hfill14,::2)
  \Bigg\}. $$
One can easily check that
$$\mathcal{\widetilde{LR}}_{\la,\mu}^{\nu} = \Bigg\{ \young(\hfill\hfill\hfill1,:\hfill23,::4)\, , \ \young(\hfill\hfill\hfill3,:\hfill14,::2) \Bigg\},$$
and hence $f^{\nu}_{\la, \mu}=2$.
From similar calculations, we obtain the following decomposition of tensor product of $\q(n)$-crystals:
\begin{align*}
 \B(3 \epsilon_1 + \epsilon_2) \otimes \B(3 \epsilon_1 + \epsilon_2) \simeq & \
 \B(6\epsilon_1 + 2\epsilon_2)
\oplus \B(5\epsilon_1 + 3\epsilon_2)^{\oplus 2} \\
 \oplus &  \, \B(5\epsilon_1 + 2\epsilon_2 + \epsilon_3)^{\oplus 2}
\oplus \B(4\epsilon_1 + 3\epsilon_2 + \epsilon_3)^{\oplus 2}.
\end{align*}
\ee
\end{example}

\end{document}